\documentclass[12pt,reqno]{amsart}
\usepackage{graphicx}
\usepackage{color}
\usepackage{amsmath,amssymb,amsthm,amsfonts}
\usepackage{graphicx}
\usepackage{color}
\usepackage{multicol}
\usepackage{bm}

\usepackage{mathptmx}
\usepackage{enumerate}
\usepackage[numbers,sort&compress]{natbib}
\numberwithin{equation}{section}
\newcommand{\R}{{\mathbb R}}

\newcommand{\be}{\begin{equation}}
\newcommand{\ee}{\end{equation}}

\newcommand{\ben}{\begin{eqnarray*}}
\newcommand{\enn}{\end{eqnarray*}}

\newcommand{\Om}{\Omega}

\newcommand{\na}{\nabla}

\newtheorem{theorem}{\textbf Theorem}[section]
\newtheorem{lemma}{\textbf Lemma}[section]
\newtheorem{rem}{\textbf Remark}[section]

\newtheorem{prop}{\textbf Proposition}[section]

\def\endProof{{\hfill$\Box$}}
\begin{document}

\author{Qianqian Hou}
\address{Institute for Advanced Study in Mathematics, Harbin Institute of Technology, Harbin 150001, People's Republic of China}
\email{qianqian.hou@hit.edu.cn}

\title[Global solvability to a two-dimensional chemotaxis-Euler system]{Global solvability to a two-dimensional chemotaxis-Euler system with Robin boundary conditions on oxygen}

\begin{abstract}
This paper is concerned with the global solvability of a chemotaxis-Euler system in bounded domains of $\mathbb{R}^2$. Completing the system with physical boundary conditions, we firstly prove local well-posedness of the corresponding initial-boundary value problem and then present some blow-up criteria. We also show that the local solutions can be extended globally in time under suitable smallness assumptions on initial and boundary data.
\end{abstract}

\subjclass[2010]{35A01, 35K57, 35Q92, 92C17}

\keywords{Chemotaxis, Euler equations, Global well-posedness, Robin boundary conditions}
\maketitle

\section{Introduction}
\subsection{Background and literature review}
Chemotaxis is termed as the oriented movement of biological species towards a higher (or lower) concentration of a chemical substance. On account of its fundamental contributions to biological processes, various mathematical models have been proposed (cf. \cite{patlak1953, keller-segel1970, KS71a}). In particular, considering that some of the chemotaxis occurs in fluid the following coupled chemotaxis-fluid model was proposed by Tuval \emph{et al.} in \cite{tuval2005} based on an oxytactic bacteria experiment:
\be\label{e1}
\left\{
\begin{array}{lll}
n_t+u\cdot\nabla n+\nabla\cdot(\chi(c)n \nabla c)=D_n \Delta n, \quad  \ \ \quad x\in \Om,\ \ t>0,
\\
c_t+u\cdot \nabla c+n f(c)=D_c\Delta c,\quad\quad \ \qquad\quad\  \, \, \,\, x\in \Om,\ \ t>0,\\
u_t+u\cdot\nabla u+\nabla p+n\nabla \phi=D_u\Delta u,\quad\quad \ \ \ \ \  \ \ x\in \Om,\ \ t>0,\\
\nabla\cdot u=0,\qquad\qquad\qquad\qquad\quad\ \ \  \ \ \ \  \quad \quad \ x\in \Om,\ \ t>0,
\end{array}
\right.
\ee
where the domain $\Omega\subset \R^{d}$ with $d\geq 1$. Here, $n(x,t)$, $c(x,t)$ denote the bacteria density and oxygen concentration, respectively; $u(x,t)$ and $p(x,t)$ represent the fluid velocity and the associated pressure. The non-negative constants $D_n$, $D_c$ and $D_u$ are diffusion coefficients. The first two equations in \eqref{e1} describe the evolution of bacteria and oxygen, where $\nabla\cdot(\chi(c)n \nabla c)$ accounts for the chemotactic response of bacteria to oxygen with chemotactic intensity $\chi(c)>0$ and $nf(c)$ stands for the oxygen consumed by bacteria with consumption rate $f(c)>0$. The last two equations are the incompressible Navier-Stokes equations when $D_u>0$ (reducing to the Euler equations when $D_u=0$) with an additional term $n\nabla \phi$ corresponding to the gravitational force exerted by the bacteria on the fluid with the given gravitational potential $\phi(x)$ independent of $t$. Model \eqref{e1} has attracted a lot of mathematical analysis since it was proposed in 2005, and we next briefly review some previous results on well-posedness. 
~\\
~\\
\textbf{Fully diffusive case.} Most of the previous studies on well-posedness of \eqref{e1} focused on the fully diffusive case, where $D_n, D_c,D_u>0$. For the Cauchy problem in $\Omega=\R^2$,  global weak solutions were constructed in \cite{duan-lorz-markowich2010}, \cite{liu-lorz2011} and uniqueness of such weak solutions were proved in \cite{chae-kang-li2012}, \cite{zhang-zheng2014} under certain structural conditions on $\chi$ and $f$. Chae-Kang-Li later established some blow-up criteria and based on these blow-up criteria they showed that local classical solutions can be extended globally under the assumptions that $\chi$ and $f$ fulfill the structural requirements stated in \cite{liu-lorz2011} or that $\|c(t=0)\|_{L^\infty}$ is sufficiently small in the prototype case $\chi(c)=1$, $f(c)=c$. For the case $\Om=\R^3$, well-posedness results available so far are merely confined to local and global small classical solutions (cf. \cite{chae-kang-li2012, chae-kang-li2014, duan-lorz-markowich2010}). In the case that $\Om$ is a bounded domain in $\R^d$, $d\geq 1$, with the influence of the surroundings on the fluid oxygen concentration neglected, Neumann boundary conditions were imposed on $n$ and $c$, under which numerous well-posedness results have been derived. Local weak solutions were given in \cite{lorz2010} when $\chi$ is a constant and $f$ is monotonically increasing with $f(0)=0$. Under some structural hypotheses on $\chi$ and $f$ that includes the prototype case $\chi=1$, $f(c)=c$, Winkler established global existence of weak solutions in the 3D case (cf. \cite{winkler2012, winkler2016}) and of smooth solutions in the 2D case (cf. \cite{winkler2012}). These solutions stabilize to the spatially uniform equilibria $(\bar{m}_0,0,\mathbf{0})$ as $t$ goes to infinity, where $\bar{m}_0=\frac{1}{|\Om|}\int_\Om m(x,0)dx$ (cf. \cite{winkler2014arma, zhang-li2015, winkler2017}). Since part of the fluid is exposed to air in the experiment \cite{tuval2005}, leading to inevitable oxygen exchange between the dissolved phase and the surrounding free phase, it is more realistic to adopt the following Robin boundary condition on $c$:
\be\label{b}
\begin{split}
\nabla c\cdot \nu=\kappa(x)(\gamma(x)-c),\qquad (\nabla n-\chi(c)n\na c)\cdot \nu=0,\qquad u=\bold{0},
\qquad x\in \Gamma,\ \ t>0,
\end{split}
\ee
where $\Gamma=\partial \Om$, $\nu(x)$ is the unit outward normal vector at $x\in \Gamma$ and $\kappa(x),\gamma(x)$ are given nonnegative functions defined on $\Gamma$. Under such boundary conditions, authors in \cite{braukhoff2017,braukhoff-tang2020} proved the global existence of classical solutions in two-dimensional case and of weak solutions in three-dimensional case. Well-posedness of nonconstant stationary solutions for the fluid-free version of \eqref{e1} in dimension $d\geq 1$ were derived in \cite{braukhoff-lankeit2019} and global classical solutions to the corresponding parabolic-elliptic system were later proved approaching to these stationary solutions in the large time limit (cf. \cite{fuest-lankeit-mizukami2021}). 
Besides the Robin boundary conditions, Dirichlet boundary conditions on oxygen have also been adopted and the study on well-posedness of solutions with such boundary conditions has been conducted in \cite{wang-xiang2021, wx2022, wang-xiang2024}. Replacing the linear cell diffusion in the first equation of \eqref{e1} with the nonlinear diffusion $\Delta n^\alpha$, $(\alpha>1)$, one derives the chemotaxis-Navier-Stokes driven by porous medium diffusion. On well-posedness of such systems we refer the reader to \cite{tao-winkler2013,tian-xiang2020, wu-xiang2020, jin2021, jin2024, tian-xiang2023, zheng-ke2022, winkler2018} and the reference therein.
~\\
~\\
\textbf{Partially or fully non-diffusive case.} Compared to the fully diffusive case, the knowledge of well-posedness for the partially or fully non-diffusive versions of \eqref{e1} is far less developed. For the fully non-diffusive case, i.e. $D_n=D_c=D_u=0$ in \eqref{e1}, Jeong-Kang established the well-posedness of local solutions in $\Omega=\mathbb{T}^d$ or $\mathbb{R}^d$ (cf. \cite{jeong-kang2022}). When $D_n, D_u>0$ and $D_c=0$, global classical solutions in $\R^2$ were derived in \cite{chae-kang-li2014} under the assumption that $\|n(t=0)\|_{L^1}<\sigma$, with the constant $\sigma$ depending on $\|c(t=0)\|_{L^\infty}$. For the chemotaxis-Euler equations, i.e. the case of $D_n, D_c>0$, $D_u=0$, Na derived global classical solutions in $\mathbb{R}^2$ provided that $\|c(t=0)\|_{L^\infty}$ is sufficiently small (cf. \cite{na2024}). For singularity formation of \eqref{e1} in the fully non-diffusive case, we refer the reader to \cite{jeong-kang2022, na20241}.
\subsection{Motivation and main results.} As mentioned above, the only well-posedness result on the chemotaxis-Euler equations is given by Na in \cite{na2024}, where the local and global classical solutions with small $\|c(t=0)\|_{L^\infty}$ are derived in the case $\Omega=\R^2$, but no blow-up criteria are established. In this paper, we aim to provide some blow-up criteria for local solutions of the chemotaxis-Euler equations in bounded domains $\Omega\subset \mathbb{R}^2$ under the physical boundary condition on $c$ in \eqref{b}. 
In line with the experiment in \cite{tuval2005}, we set $\chi(c)=1$, $f(c)=c$ in \eqref{e1} and study the following initial boundary value problem:
\be\label{cns}
\left\{
\begin{array}{lll}
n_t+u\cdot\nabla n+\nabla\cdot(n \nabla c)= \Delta n, \quad  \ \ \quad x\in \Om,\ \ t>0,
\\
c_t+u\cdot \nabla c+n c=\Delta c,\quad\quad \ \qquad\quad\   x\in \Om,\ \ t>0,\\
u_t+u\cdot\nabla u+\nabla p=-n\nabla \phi,\quad\quad \ \quad \ x\in \Om,\ \ t>0,\\
\nabla\cdot u=0,\qquad\qquad\qquad\qquad \quad \quad \ x\in \Om,\ \ t>0\\
\end{array}
\right.
\ee
with
\be\label{bc}
\begin{split}
&n(x,0)=n_0(x),\qquad  c(x,0)=c_0(x),\qquad u(x,0)=u_0(x),\qquad x\in \Om,\\
\nabla c\cdot \nu&=\kappa(x)(\gamma(x)-c),\qquad (\nabla n-n\na c)\cdot \nu=0,\qquad u\cdot\nu=0,
\qquad x\in \Gamma,\ \ t>0,
\end{split}
\ee
where we have taken $D_n=D_c=1$ without loss of generality, $\Om\subset \mathbb{R}^2$ is a bounded domain with smooth boundary $\Gamma=\partial \Om$, $\nu(x)$ is the unit outward normal vector at $x\in \Gamma$ and $\kappa(x),\gamma(x)$ are given nonnegative functions defined on $\Gamma$. The no-penetration boundary condition is imposed on $u$, since we are considering the Euler equations.

Our first result concerns the local well-posedness of \eqref{cns}-\eqref{bc}. Its proof is based on Banach's fixed point theorem and will be postponed to the appendix.
\begin{prop}\label{p1} Let 
$n_0, c_0\in H^2(\Om), u_0\in H^3(\Om), \kappa\in H^{\frac{3}{2}}(\Gamma), \gamma\in H^{\frac{5}{2}}(\Gamma)$ and $\na\phi\in W^{3,\infty}(\Omega)$
with  $n_0, c_0, \kappa, \gamma \geq 0$ satisfying the following compatibility conditions:
\ben
\begin{split}
\nabla c_0\cdot \nu=\kappa(x)(\gamma(x)-c_0),\ \  (\nabla n_0-n_0\na c_0)\cdot \nu=0,\ \  u_0\cdot\nu=0\ \ \text{on}\ \ \Gamma;\quad \na\cdot u_0=0\ \ \text{in}\ \ \Omega.
\end{split}
\enn
Then there exists a $T_{loc}>0$ such that \eqref{cns}-\eqref{bc} admits a unique solution $(n,c,u,\na p)$ with $n, c\geq 0$ fulfilling
\be\label{b11}
\begin{split}
&(n,c, u,\na p)\in C([0,T_{loc}]; H^2(\Omega)\times H^2(\Omega)\times H^3(\Omega)\times H^2(\Omega)),\\
&\ \ \ \quad\  (n,c,\na p)\in L^2(0,T_{loc}; H^3(\Omega)\times H^3(\Omega)\times H^3(\Omega)).
\end{split}
\ee
\end{prop}
We next present our main result, a blow-up criterion on the local solutions derived in Proposition \ref{p1}.
\begin{theorem}\label{t1}
 Let the assumptions in Proposition \ref{p1} hold. If $T^*$, the maximal time of existence in Proposition \ref{p1}, is finite, then  
 \be\label{a00}
 \|n\|_{L^q(0,T^*;L^p(\Omega))}=\infty,\quad \frac{2}{p}+\frac{2}{q}\leq 1,\quad 2<p\leq \infty.
 \ee
\end{theorem}
As a consequence of Theorem \ref{t1}, we show that the local solutions given in Proposition \ref{p1} can be extended globally, under a suitable smallness condition on the initial and boundary data.
\begin{theorem}\label{t2}
 Let the assumptions in Proposition \ref{p1} hold. Assume further that
 \be\label{a0}
 \max\{\|\gamma\|_{L^\infty(\Gamma)},\|c_0\|_{L^\infty(\Om)}\}\leq \frac{1}{48}.
 \ee
 Then the unique solution derived in Proposition \ref{p1} exists globally in time, that is, for each $T>0$
 \ben
\begin{split}
&(n,c, u,\na p)\in C([0,T]; H^2(\Omega)\times H^2(\Omega)\times H^3(\Omega)\times H^2(\Omega)),\\
&\ \ \ \quad\ (n,c, \na p)\in L^2(0,T; H^3(\Omega)\times H^3(\Omega)\times H^3(\Omega)).
\end{split}
\enn
\end{theorem}
\begin{rem}
Our result in Theorem \ref{t1} may be generalized to the case $\Omega=\R^2$
by slightly modifying our proof. This means that we obtain a blow-up criterion for the local solutions given in \cite[Theorem 1.1]{na2024}.
\end{rem}

Compared with classical chemotaxis models, the main difficulty in establishing the well-posedness of the coupled chemotaxis-fluid system \eqref{e1} lies in the convective transport of the oxygen by the fluid. Specifically, to control the chemotactic term $\nabla\cdot(\chi(c)n\nabla c)$ in the first equation, one is naturally led to apply the gradient operator to the second equation, thereby seeking estimates for $\na c$. This procedure, however, inevitably introduces the strongly nonlinear term $\na(u\cdot \na c)$, which requires careful handling. In the chemotaxis-Navier–Stokes system setting, this obstacle is effectively overcome by the regularizing effect provided by the dissipative term $D_u \Delta u$ in the fluid equation. By contrast, in the chemotaxis–Euler system, the lack of such dissipation renders the treatment of this term considerably more challenging. For the whole-space case, Na in \cite{na2024} successfully resolved this issue by employing the vanishing viscosity method (i.e., adding a small artificial viscosity in the fluid equation and then passing to the vanishing viscosity limit) together with a Brezis-Gallouet-Wainger type inequality. Nevertheless, since the vanishing viscosity process generates boundary layer effects in domains with a boundary, this approach is no longer applicable to the problem considered in the present work.

In view of this limitation, we adopt a different approach that works directly with the chemotaxis–Euler equations. We first establish the local well‑posedness of the system in the appendix. In Section 2, employing the energy method, Neumann heat semigroup, a series of Sobolev embedding inequalities and Brezis-Gallouet-Wainger type inequalities, we carefully handle the boundary terms and overcome the absence of dissipation in the fluid equation, and successively prove Theorem \ref{t1}. Finally, in Section 3, we extend the local solution to a global one, under the smallness assumption \eqref{a0}. 
~\\
~\\
\textbf{Notations.} For convenience, we introduce some notations that will be used in the remaining part of this paper. We denote by $C_0$ and $C_{T}$ generic constants that may change from one line to another, where $C_0$ depends only on $\Omega$ and the initial and boundary data, while $C_{T}$ may additionally depend on $T$.
  $L^q_x$ with $1\leq q\leq \infty$ represents $L^q(\Omega)$ and $H^k_x$ with $k\geq 1$ represents $H^k(\Omega)$. $L^q_{T} \boldsymbol{X}$ with $1\leq q\leq \infty$ stands for $L^q(0,T; \boldsymbol{X})$ for Banach spaces $\boldsymbol{X}$.

\section{Proof of Theorem \ref{t1}}

The proof of Theorem \ref{t1} will be given at the end of this section, based on a series of results established in the following Lemmas \ref{l0}--\ref{l7}. Lemmas \ref{l2}--\ref{l7} are devoted to justifying that the regularity result in \eqref{b11} can be extended up to $T^*$, provided that $T^*<\infty$ and \eqref{a45} holds true. Lemma \ref{l0} is to handle the boundary terms induced by the inhomogeneous boundary conditions in \eqref{bc}.

\begin{lemma}\label{l0}
Let the assumptions in Theorem \ref{t1} hold. Then there exists a constant $C_0$ such that
\be\label{a11}
\begin{split}
\frac{1}{2}&\int_{\Gamma}\partial_\nu(\partial_\nu c)^2\,dS+\int_{\Gamma}\kappa|\na_{\Gamma} c|^2\,dS+\frac{1}{2}\frac{d}{dt}\int_{\Gamma}\kappa c^2\,dS\\
\leq &\frac{1}{4}\|\na^2 c\|_{L^2_x}^2+\frac{1}{6}\|\na n\|_{L^2_x}^2+C_0(\|u\|_{H^1_x}^2+\|\na c\|_{L^2_x}^2+\|n\|_{L^2_x}^2+1)+\frac{d}{dt}\int_{\Gamma}\kappa \gamma c\,dS
\end{split}
\ee
for all $t\in (0,T^*)$, where $\na_{\Gamma} c$ is the surface gradient of $c$ on $\Gamma$.
\end{lemma}
\begin{proof}
Firstly, the maximum principle and the Sobolev embedding inequality give
\be\label{a36}
\|c\|_{L^\infty_{T^*}L^\infty_x}\leq \max\{\|\gamma\|_{L^\infty(\Gamma)}, \|c_0\|_{L^\infty_x}\}\leq \max\{C_0\|\gamma\|_{H^1(\Gamma)}, \|c_0\|_{L^\infty_x}\}\leq C_0,
\ee
where the generic constant $C_0$ depends only on $\Omega$ and the initial and boundary data.
We next recall the relation between the Laplace operator and Laplace-Beltrami operator on $\Gamma$ (see e.g. \cite[Lemma 1]{xu-zhao2003}):
\ben
\begin{split} 
\Delta_{\Gamma}c=\Delta c-(\na\cdot \nu) \partial_{\nu}c-\partial_{\nu}(\partial_{\nu} c)+\nu (\na \nu)\na c,
\end{split}
\enn
which, along with the second equation in \eqref{cns} entails that
\ben
c_t=\Delta_{\Gamma} c+(\na\cdot \nu) \partial_{\nu}c+\partial_{\nu}(\partial_{\nu} c)-\nu (\na \nu)\na c-u\cdot\na c-nc\quad \text{on}\,\,\,\,\Gamma.
\enn
Multiplying the above equality with $\kappa(x)(\gamma(x)-c)$ in $L^2(\Gamma)$ and using the boundary condition $\partial_\nu c=\kappa(x)(\gamma(x)-c)$ on $\Gamma$, we obtain
\be\label{a7}
\begin{split}
\frac{1}{2}&\int_{\Gamma}\partial_\nu(\partial_\nu c)^2\,dS+\frac{1}{2}\frac{d}{dt}\int_{\Gamma}\kappa c^2\,dS\\
=&-\int_{\Gamma}\Delta_{\Gamma}c\, \kappa(\gamma-c)\,dS-\int_{\Gamma}(\na\cdot\nu)\kappa^2(\gamma-c)^2\,dS
+\int_{\Gamma}\nu(\na \nu)\na c \kappa(\gamma-c)\,dS\\
&+\int_{\Gamma}(u\cdot\na c)\kappa(\gamma-c)\,dS+\int_{\Gamma}nc\kappa(\gamma-c)\,dS+\frac{d}{dt}\int_{\Gamma}\kappa \gamma c\,dS.
\end{split}
\ee
By the trace theorem and \eqref{a36}, one gets
\ben
\begin{split}
-\int_{\Gamma}\Delta_{\Gamma}c \kappa(\gamma-c)\,dS
=&\int_{\Gamma}\na_{\Gamma}c \na_{\Gamma}(\kappa\gamma)\,dS
-\int_{\Gamma}\na_{\Gamma}c \na_{\Gamma}(\kappa c)\,dS\\
=&\int_{\Gamma}\na_{\Gamma}c \na_{\Gamma}(\kappa\gamma)\,dS
-\int_{\Gamma}\na_{\Gamma}c \na_{\Gamma}\kappa c\,dS
-\int_{\Gamma}\kappa|\na_{\Gamma}c|^2\,dS\\
\leq &\|\na_{\Gamma}c\|_{L^2(\Gamma)}(\|\na_{\Gamma}(\gamma\kappa)\|_{L^2(\Gamma)}
+\|\na_{\Gamma}\kappa\|_{L^2(\Gamma)}\|c\|_{L^\infty(\Gamma)})-\int_{\Gamma}\kappa|\na_{\Gamma}c|^2\,dS\\
\leq& C_0(\|\na c\|_{L^2_x}+\|\na^2 c\|_{L^2_x})\|\kappa\|_{H^1(\Gamma)}(\|\gamma\|_{H^1(\Gamma)}+\|c\|_{L^\infty_x})-\int_{\Gamma}\kappa|\na_{\Gamma}c|^2\,dS\\
\leq &\frac{1}{12}\|\na^2 c\|_{L^2_x}^2+C_0(\|\na c\|_{L^2_x}^2+1)-\int_{\Gamma}\kappa|\na_{\Gamma}c|^2\,dS.
\end{split}
\enn
The trace theorem and \eqref{a36} further entails that
\ben
\begin{split}
\int_{\Gamma}\nu(\na \nu)\na c \kappa(\gamma-c)\,dS
\leq& \|\nu(\na \nu)\|_{L^\infty(\Gamma)}\|\na c\|_{L^2(\Gamma)}(\|\kappa\gamma\|_{L^2(\Gamma)}+\|\kappa c\|_{L^2(\Gamma)})\\
\leq &C_0(\|\na c\|_{L^2_x}+\|\na^2 c\|_{L^2_x})\|\kappa\|_{H^1(\Gamma)}(\|\gamma\|_{H^1(\Gamma)}+\|c\|_{L^\infty_x})\\
\leq &\frac{1}{12}\|\na^2 c\|_{L^2_x}^2+C_0(\|\na c\|_{L^2_x}^2+1).
\end{split}
\enn
Similarly, 
\ben
\begin{split}
&-\int_{\Gamma}(\na\cdot\nu)\kappa^2(\gamma-c)^2\,dS
+\int_{\Gamma}(u\cdot\na c)\kappa(\gamma-c)\,dS+\int_{\Gamma}nc\kappa(\gamma-c)\,dS\\
\leq&C_0\|\na\cdot\nu\|_{L^\infty(\Gamma)}\|\kappa\|_{L^2(\Gamma)}^2(\|\gamma\|_{L^\infty(\Gamma)}^2+\|c\|_{L^\infty(\Gamma)}^2)\\
&+C_0\|\kappa\|_{L^\infty(\Gamma)}(\|u\|_{L^2(\Gamma)}\|\na c\|_{L^2(\Gamma)}+\|n\|_{L^2(\Gamma)}\| c\|_{L^2(\Gamma)})(\|\gamma\|_{L^\infty(\Gamma)}+\|c\|_{L^\infty(\Gamma)})\\
\leq &C_0\|u\|_{H^1_x}(\|\na c\|_{L^2_x}+\|\na^2 c\|_{L^2_x})+C_0(\|n\|_{H^1_x}+1)\\
\leq &\frac{1}{12}\|\na^2 c\|_{L^2_x}^2+\frac{1}{6}\|\na n\|_{L^2_x}^2+C_0(\|u\|_{H^1_x}^2+\|\na c\|_{L^2_x}^2+\|n\|_{L^2_x}^2+1).
\end{split}
\enn
Substituting the above estimates into \eqref{a7}, one derives the desired estimates. The proof is finished.

\end{proof}
\begin{lemma}\label{l2}
Suppose that the assumptions in Theorem \ref{t1} hold true. Then there exists a constant $C_0$ such that
\ben
\begin{split}
&\frac{d}{dt}(\|u\|_{L^2_x}^2+\|\omega\|_{L^2_x}^2+\|\na c\|_{L^2_x}^2)+\frac{1}{2}\frac{d}{dt}\int_{\Gamma}\kappa c^2\,dS+\frac{5}{4}\|\na^2 c\|_{L^2_x}^2\\
\leq & \frac{1}{2}\|\na n\|_{L^2_x}^2+
C_0(\|n\|_{L^2_x}^2+\|u\|_{L^2_x}^2+\|\na c\|_{L^2_x}^2+\|\omega\|_{L^2_x}^2+1)
+\frac{d}{dt}\int_{\Gamma}\kappa \gamma c\,dS
\end{split}
\enn
for all $t\in (0,T^*)$, where $\omega=\na\times u$.
\end{lemma}
\begin{proof}
Testing the third equation of \eqref{cns} with $u$ in $L^2(\Omega)$ and using the fact $\na\cdot u=0$ in $\Omega$, one gets
\be\label{a8}
\frac{d}{dt}\|u\|_{L^2_x}^2\leq C_0(\|n\|_{L^2_x}^2+\|u\|_{L^2_x}^2),
\ee
where the constant $C_0$ depends only on $\phi$.
Applying $\na\times$ to the third equation of \eqref{cns}, one deduces the equation for vorticity $\omega=\na\times u$: 
\be\label{a26}
\omega_t+u\cdot \na\omega=(\na\times n) \cdot(\na^{\perp} \phi),
\ee
which gives
\be\label{a37}
\frac{d}{dt}\|\omega\|_{L^2_x}^2\leq \|\na n\|_{L^2_x}\|\na \phi\|_{L^\infty_x}\|\omega\|_{L^2_x}\leq \frac{1}{6}\|\na n\|_{L^2_x}^2+C_0\|\omega\|_{L^2_x}^2.
\ee
From the second equation of \eqref{cns}, we know that the equation for $\na c$ is as follows:
\ben
\na c_t+\na(u\cdot \na c)=\Delta\na c-\na nc-n\na c,
\enn
which, multiplied with $\na c$ in $L^2(\Omega)$ along with integration by parts yields
\be\label{a10}
\begin{split}
\frac{1}{2}\frac{d}{dt}\|\na c\|_{L^2_x}^2+\|\na^2 c\|_{L^2_x}^2+\|\sqrt{n}\na c\|_{L^2_x}^2
=&-\int_{\Omega}\na(u\cdot \na c)\cdot\na c\,dx-\int_{\Omega}c\na n\cdot\na c\,dx\\
&+\int_{\Gamma}[(\partial_1\na c\cdot \nu)\partial_1 c+(\partial_2\na c\cdot \nu)\partial_2 c]\,dS\\
=:&\,I_1+I_2+I_3.
\end{split}
\ee
From the facts $\na \cdot u=0$ in $\Omega$, $u\cdot\nu=0$ on $\Gamma$, Gagliardo-Nirenberg inequality, Cauchy-Schwarz inequality and \eqref{a36}, we deduce that
\ben
\begin{split}
I_1
=&-\frac{1}{2}\int_{\Omega}u\cdot \na(|\na c|^2)\,dx-\int_{\Omega}\na c \na u\na c\,dx\\
\leq &\|\na u\|_{L^2_x}\|\na c\|_{L^4_x}^2\\
\leq &\|\na u\|_{L^2_x}\|c\|_{L^\infty_x}\|\na^2 c\|_{L^2_x}\\
\leq &\frac{1}{16}\|\na^2 c\|_{L^2_x}^2+C_0\|\na u\|_{L^2_x}^2.
\end{split}
\enn
The Cauchy-Schwarz inequality and \eqref{a36} lead to
\ben
I_2\leq\|c\|_{L^\infty_x}\|\na n\|_{L^2_x}\|\na c\|_{L^2_x}\leq  \frac{1}{6}\|\na n\|_{L^2_x}^2+C_0\|\na c\|_{L^2_x}^2.
\enn
Let $\tau$ be the unit tangent vector of $\Gamma$ at $x$. Then the boundary condition for $c$ in \eqref{bc}, the fact $\kappa\geq 0$, \eqref{a36} and the trace theorem imply that
\ben
\begin{split}
\int_{\Gamma}[\na (\na c\cdot\nu)]\cdot \na c\,dS
=&\int_{\Gamma}[\na (\na c\cdot\nu)]\cdot [(\na c\cdot\tau)\tau+(\na c\cdot\nu)\nu]\,dS\\
=&\int_{\Gamma}(\na c\cdot\tau)\partial_{\tau}[k(\gamma-c)]\,dS+\int_{\Gamma}(\na c\cdot\nu)\nu\cdot\na (\na c\cdot\nu)\,dS\\
=&\int_{\Gamma}\partial_{\tau}c\partial_{\tau}(\kappa\gamma)\,dS-\int_{\Gamma}\kappa(\partial_{\tau}c)^2\,dS
-\int_{\Gamma}c\partial_{\tau}c\partial_{\tau}\kappa\,dS+\frac{1}{2}\int_{\Gamma}\partial_\nu(\partial_\nu c)^2\,dS\\
\leq &\|\na c\|_{L^2(\Gamma)}(\|\kappa\gamma\|_{H^1(\Gamma)}+\|c\|_{L^\infty_x}\|\kappa\|_{H^1(\Gamma)})+\frac{1}{2}\int_{\Gamma}\partial_\nu(\partial_\nu c)^2\,dS\\
\leq &C_0(\|\na c\|_{L^2_x}+\|\na^2 c\|_{L^2_x})(\|\kappa\gamma\|_{H^1(\Gamma)}+\|c\|_{L^\infty_x}\|\kappa\|_{H^1(\Gamma)})+\frac{1}{2}\int_{\Gamma}\partial_\nu(\partial_\nu c)^2\,dS\\
\leq&\frac{1}{32}\|\na^2c\|_{L^2_x}^2+C_0(\|\na c\|_{L^2_x}^2+1)+\frac{1}{2}\int_{\Gamma}\partial_\nu(\partial_\nu c)^2\,dS.
\end{split}
\enn
Substituting the above estimate into $I_3$ and using the trace theorem, one gets
\ben
\begin{split}
I_3=&\int_{\Gamma}[\na (\na c\cdot\nu)]\cdot \na c\,dS-\int_{\Gamma}[(\na c\cdot \partial_1\nu)\partial_1 c+(\na c\cdot \partial_2\nu)\partial_2 c]\,dS\\
\leq &\int_{\Gamma}[\na (\na c\cdot\nu)]\cdot \na c\,dS+C_0\|\na c\|_{L^2(\Gamma)}^2\|\na\nu\|_{L^\infty(\Gamma)}\\
\leq &\int_{\Gamma}[\na (\na c\cdot\nu)]\cdot \na c\,dS+C_0\|\na c\|_{L^2_x}\|\na^2 c\|_{L^2_x}\\
\leq &\frac{1}{16}\|\na^2c\|_{L^2_x}^2+C_0(\|\na c\|_{L^2_x}^2+1)+\frac{1}{2}\int_{\Gamma}\partial_\nu(\partial_\nu c)^2\,dS.
\end{split}
\enn
Substituting the above estimates for $I_1$, $I_2$ and $I_3$ into \eqref{a10} and using \eqref{a11}, we obtain
\ben
\begin{split}
\frac{d}{dt}\|\na c\|_{L^2_x}^2+\frac{5}{4}\|\na^2 c\|_{L^2_x}^2+\|\sqrt{n}\na c\|_{L^2_x}^2
\leq & \frac{1}{3}\|\na n\|_{L^2_x}^2+
C_0(\|n\|_{L^2_x}^2+\|\na c\|_{L^2_x}^2+\|\na u\|_{L^2_x}^2+1)\\
&+\frac{d}{dt}\int_{\Gamma}\kappa \gamma c\,dS
-\frac{1}{2}\frac{d}{dt}\int_{\Gamma}\kappa c^2\,dS,
\end{split}
\enn
which, in conjunction \eqref{a8}, \eqref{a37} and the Calderón-Zygmund inequality $\|\na u\|_{L^2_x}\leq C_0\|\omega\|_{L^2_x}$  (see e.g. \cite[Proposition 7.5]{Bahouri-Chemin-Dachin2011}), completes the proof.

\end{proof}
\begin{lemma}\label{l8}
Let the assumptions in Theorem \ref{t1} hold true. Assume further that $T^*<\infty$ and that
\be\label{a45}
\|n\|_{L^q(0,T^*;L^p(\Omega))}<\infty,\qquad \frac{2}{p}+\frac{2}{q}\leq 1
\ee
with some $2<p\leq \infty$. Then there exists a constant $C_{T^*}$ depending on $T^*$, such that
\be\label{a38}
\|n\|_{L^\infty_{T^*}L^2_x}+\|\na c\|_{L^\infty_{T^*}L^2_x}+\|u\|_{L^\infty_{T^*}H^1_x}
+\|\na n\|_{L^2_{T^*}L^2_x}+\|\na^2 c\|_{L^2_{T^*}L^2_x}\leq C_{T^*}.
\ee
\end{lemma}
\begin{proof}
We split the proof into the case $2<p<\infty$ and the case $p=\infty$.
~\\
$\bullet$ \textbf{The case} {\boldmath $2<p<\infty$}. Taking the $L^2_x$ inner product of the first equation in \eqref{cns} and using integration by parts to have
\be\label{a39}
\begin{split}
\frac{1}{2}\frac{d}{dt}\|n\|_{L^2_x}^2+\|\na n\|_{L^2_x}^2
=&\int_{\Omega} n \na c\cdot \na n\,dx\\
\leq& \|n\|_{L^p_x}\|\na n\|_{L^2_x}\|\na c\|_{L^{2p/(p-2)}_x}\\
\leq &C_0\|n\|_{L^p_x}\|\na n\|_{L^2_x}\|c\|_{L^{\infty}_x}^{(1-2/p)}\|\na^2 c\|_{L^2_x}^{2/p}\\
\leq &\frac{1}{4}\|\na n\|_{L^2_x}^2+\frac{1}{8}\|\na^2 c\|_{L^2_x}^2+C_p\|n\|_{L^p_x}^{2p/(p-2)},
\end{split}
\ee
where in the second inequality we used the Gagliardo-Nirenberg inequality and in the last inequality we used \eqref{a36}, and the constant $C_p$ depends on $p$, $\Omega$ and the initial and boundary data. Combining the above estimates with Lemma \ref{l2}, one gets
\be\label{a46}
\begin{split}
&\frac{d}{dt}(\|u\|_{L^2_x}^2+\|\omega\|_{L^2_x}^2+\|\na c\|_{L^2_x}^2+\|n\|_{L^2_x}^2+\frac{1}{2}\int_{\Gamma}\kappa c^2\,dS)+\|\na^2 c\|_{L^2_x}^2+\|\na n\|_{L^2_x}^2\\
\leq & 
C_0(\|n\|_{L^2_x}^2+\|u\|_{L^2_x}^2+\|\na c\|_{L^2_x}^2+\|\omega\|_{L^2_x}^2+1)
+\frac{d}{dt}\int_{\Gamma}\kappa \gamma c\,dS+C_p\|n\|_{L^p_x}^{2p/(p-2)}.
\end{split}
\ee
Noting that 
\ben
\int_0^{T^*}\|n\|_{L^p_x}^{2p/(p-2)}\,dt\leq (T^*)^{1-2p/[q(p-2)]} \|n\|_{L^q_{T^*}L^p_x}^{2p/(p-2)}<\infty,
\enn
we apply Gronwall's inequality to \eqref{a46} and use \eqref{a36} and the Calderón-Zygmund inequality $\|\na u\|_{L^2_x}\leq C_0 \|\omega\|_{L^2_x}$ to derive \eqref{a38}.
~\\
$\bullet$ \textbf{The case} {\boldmath $p=\infty$}. Similarly to the derivation of \eqref{a39}, one gets
\ben
\begin{split}
\frac{1}{2}\frac{d}{dt}\|n\|_{L^2_x}^2+\|\na n\|_{L^2_x}^2
=&\int_{\Omega} n \na c\cdot \na n\,dx\\
\leq& \|n\|_{L^\infty_x}\|\na n\|_{L^2_x}\|\na c\|_{L^2_x}\\
\leq &\frac{1}{4}\|\na n\|_{L^2_x}^2+\|n\|_{L^\infty_x}^2\|\na c\|_{L^2_x}^2,
\end{split}
\enn
which, along with Lemma \ref{l2} gives
\ben
\begin{split}
&\frac{d}{dt}(\|u\|_{L^2_x}^2+\|\omega\|_{L^2_x}^2+\|\na c\|_{L^2_x}^2+\|n\|_{L^2_x}^2+\frac{1}{2}\int_{\Gamma}\kappa c^2\,dS)+\|\na^2 c\|_{L^2_x}^2+\|\na n\|_{L^2_x}^2\\
\leq & 
C_0(\|n\|_{L^2_x}^2+\|u\|_{L^2_x}^2+\|\na c\|_{L^2_x}^2+\|\omega\|_{L^2_x}^2+1)
+\frac{d}{dt}\int_{\Gamma}\kappa \gamma c\,dS+\|n\|_{L^\infty_x}^2\|\na c\|_{L^2_x}^2.
\end{split}
\enn
Then \eqref{a38} follows from Gronwall's inequality, the assumption that $\|n\|_{L^2_{T^*}L^\infty_x}<\infty$, \eqref{a36} and the Calderón-Zygmund inequality $\|\na u\|_{L^2_x}\leq C_0\|\omega\|_{L^2_x}$.

\end{proof}
\begin{lemma}\label{l3}
Let the assumptions in Lemma \ref{l8} hold true. Then there exists a constant $C_{r,\rho,T^*}$ depending on $2\leq r<\infty$, $1<\rho<\infty$, $T^*$, $\Omega$ and the initial and boundary data, such that
\ben
\|n\|_{L^\infty_{T^*}L^r_x}+\|\na c\|_{L^\infty_{T^*}L^r_x}+\|c\|_{L^\rho_{T^*}W^{2,r}_x}\leq C_{r,\rho,T^*}.
\enn 
\end{lemma}
\begin{proof}
First, it follows from the Gagliardo-Nirenberg inequality and Lemma \ref{l8} that
\be\label{a13}
\begin{split}
\|\na c\|_{L^4_{T^*}L^4_x}\leq \|\na c\|_{L^\infty_{T^*}L^2_x}^{\frac{1}{2}}\|\na^2 c\|_{L^2_{T^*}L^2_x}^{\frac{1}{2}}\leq C_{T^*}.
\end{split}
\ee
Then taking the $L^2_x$ inner product of the first equation in \eqref{cns} with $n^{r-1}$ and using integration by parts, we obtain
\ben
\begin{split}
\frac{1}{r}\frac{d}{dt}\int_{\Omega} n^r\,dx+(r-1)\int_{\Omega}|\na n|^2n^{r-2}\,dx
=&(r-1)\int_{\Omega} n^{r-1}\na c\cdot \na n\,dx\\
\leq& \frac{r-1}{2}\int_{\Omega}|\na n|^2n^{r-2}\,dx+\frac{r-1}{2}\int_{\Omega}|\na c|^2n^{r}\,dx,
\end{split}
\enn
where 
\ben
\begin{split}
\int_{\Omega}|\na c|^2n^{r}\,dx
\leq& \|n^{\frac{r}{2}}\|_{L^4_x}^2\|\na c\|_{L^4_x}^2\\
\leq& C_0(\|\na n^{\frac{r}{2}}\|_{L^2_x}\|n^{\frac{r}{2}}\|_{L^2_x}+\|n\|_{L^2_x})\|\na c\|_{L^4_x}^2\\
=&C_0(\frac{r}{2}\|n^{\frac{r-2}{2}}\na n \|_{L^2_x}\|n\|_{L^r_x}^{\frac{r}{2}}+\|n\|_{L^2_x})\|\na c\|_{L^4_x}^2\\
\leq &\int_{\Omega}|\na n|^2n^{r-2}\,dx+C_0 r^2 \|\na c\|_{L^4_x}^4\|n\|_{L^r_x}^r+C_0\|n\|_{L^2_x}\|\na c\|_{L^4_x}^2.
\end{split}
\enn
Hence 
\ben
\begin{split}
\frac{1}{r}\frac{d}{dt}\int_{\Omega} n^r\,dx
\leq &C_r  (\|\na c\|_{L^4_x}^4\|n\|_{L^r_x}^r+\|n\|_{L^2_x}\|\na c\|_{L^4_x}^2),
\end{split}
\enn
where the constant $C_r$ depends on $r$.
Applying Gronwall's inequality to the above estimates and using \eqref{a13} and Lemma \ref{l8}, one gets
\be\label{a14}
\|n\|_{L^\infty_{T^*}L^r_x}\leq C_{r,T^*},\qquad \forall\ r\geq 2,
\ee
with the constant $C_{r,T^*}$ depending on $r$ and $T^*$.
We proceed to estimating $\|\na c\|_{L^\infty_{T^*}L^r_x}$. By the trace theorem (see e.g. \cite[Theorem 8.3]{lions&magenes}), there exists $g_1, g_2\in H^{3}(\Omega)$ satisfying
\be\label{b01}
\frac{\partial g_1}{\partial \nu}(x)=\kappa(x),\quad g_2(x)=\gamma(x)\ \ \text{with}\ \ \frac{\partial g_2}{\partial\nu}(x)=0 \ \ \text{on}\ \ \Gamma. 
\ee
and
\be\label{a15}
\|g_1\|_{H^3_x}\leq C_0\|\kappa\|_{H^{\frac{3}{2}}(\Gamma)}\leq C_0,\qquad\|g_2\|_{H^3_x}\leq C_0\|\gamma\|_{H^{\frac{5}{2}}(\Gamma)}\leq C_0.
\ee
Let $\tilde{c}(x,t)=e^{g_1(x)}(g_2(x)-c(x,t))$. Then $\tilde{c}$ solves
\ben
\left\{
\begin{array}{lll}
\tilde{c}_t=\Delta \tilde{c}+(u\cdot\na c+nc-\Delta g_2)e^{g_1}+|\na g_1|^2\tilde{c}-\Delta g_1 \tilde{c}-2\na g_1\cdot \na\tilde{c},\quad   x\in \Om,\ t>0,\\
\tilde{c}(x,0)=\tilde{c}_0:=e^{g_1(x)}(g_2(x)-c_0(x)),\quad x\in \Omega,\\
\frac{\partial \tilde{c}}{\partial \nu}=0,\quad x\in \Gamma,\ \ t>0
\end{array}
\right.
\enn
and thus
\ben
\begin{split}
\tilde{c}(\cdot, t)=&e^{t\Delta}\tilde{c}_0+\int_0^t e^{(t-s)\Delta}[(u\cdot\na c+nc-\Delta g_2)e^{g_1}](\cdot,s)\,ds\\
&+\int_0^t e^{(t-s)\Delta}(|\na g_1|^2\tilde{c}-\Delta g_1 \tilde{c}-2\na g_1\cdot \na\tilde{c})(\cdot,s)\,ds,
\end{split}
\enn
where $(e^{t\Delta})_{t\geq 0}$ is the Neumann heat semigroup.  
Applying Lemma 1.3(ii)-(iii) in \cite{winkler2010} and using \eqref{a15}, one gets
\be\label{a16}
\begin{split}
&\|\na\tilde{c}(\cdot,t)\|_{L^r_x}\\
\leq& \|\na \tilde{c}_0\|_{L^r_x}+\int_0^t (t-s)^{-\frac{1}{2}-(\frac{1}{2}-\frac{1}{r})}(\|(u\cdot\na c)(\cdot,s)\|_{L^2_x}+\|(nc)(\cdot,s)\|_{L^2_x}+C_0)\,ds\\
&+C_0\int_0^t (t-s)^{-\frac{1}{2}-(\frac{1}{2}-\frac{1}{r})}(\|\tilde{c}(\cdot,s)\|_{L^2_x}+\|\na \tilde{c}(\cdot,s)\|_{L^2_x})\,ds\\
\end{split}
\ee
for $r\geq 2$, $t\in (0,T^*]$. 
The Sobolev embedding inequality, \eqref{a36} and Lemma \ref{l8} entail that
\be\label{a17}
\begin{split}
\|(u\cdot\na c)(\cdot,s)\|_{L^2_x}\leq &\|u\|_{L^\infty_{T^*}L^4_x}\|\na c(\cdot,s)\|_{L^4_x}\\
\leq& C_0\|u\|_{L^\infty_{T^*}H^1_x}\|\na c(\cdot,s)\|_{L^4_x}\\
\leq& C_{T^*}\|\na c(\cdot,s)\|_{L^4_x},\qquad \forall\ s\in (0,T^*]
\end{split}
\ee
and that
\be\label{a19}
\begin{split}
 &\|(nc)(\cdot,s)\|_{L^2_x}\leq \|n\|_{L^\infty_{T^*}L^2_x}\|c\|_{L^\infty_{T^*}L^\infty_x}\leq C_{T^*},\qquad \forall\ s\in (0,T^*].
\end{split}
\ee
It follows from Lemma \ref{l8}, the definition of $\tilde{c}$ and \eqref{a15} that
\be\label{a43}
\begin{split}
\|\tilde{c}(\cdot,s)\|_{L^2_x}+\|\na \tilde{c}(\cdot,s)\|_{L^2_x}\leq C_0(\|c(\cdot,s)\|_{L^2_x}+\|\na c(\cdot,s)\|_{L^2_x}+1)
\leq C_{T^*},\qquad \forall\ s\in (0,T^*].
\end{split}
\ee
Setting $r=3$ in \eqref{a16} and using \eqref{a17}, \eqref{a19}, \eqref{a43}, \eqref{a15} and \eqref{a13}, we have
\be\label{a18}
\begin{split}
\|\na\tilde{c}(\cdot,t)\|_{L^3_x}\leq &C_0(1+\|c_0\|_{W^{1,3}_x})+C_{T^*}\int_0^t (t-s)^{-\frac{2}{3}}(\|\na c(\cdot,s)\|_{L^4_x}+1)\,ds\\
\leq &C_0(1+\|c_0\|_{H^2_x})+C_{T^*}\|\na c\|_{L^4_{T^*}L^4_x}\big[\int_0^{t} (t-s)^{-\frac{8}{9}}\,ds\big]^{3/4}\\
\leq& C_{T^*},\qquad \forall\ t\in (0,T^*]
\end{split}
\ee
which, in conjunction with the definition of $\tilde{c}$, \eqref{a15}, Lemma \ref{l8} and the Sobolev embedding inequality gives
\be\label{a20}
\begin{split}
\|(u\cdot\na c)(\cdot,s)\|_{L^2_x}\leq& \|u\|_{L^\infty_{T^*}L^6_x}\|\na c\|_{L^\infty_{T^*}L^3_x}\\
\leq& C_0\|u\|_{L^\infty_{T^*}H^1_x}(\|\na \tilde{c}\|_{L^\infty_{T^*}L^3_x}+1)\\
\leq& C_{T^*},\qquad \forall\ s\in (0,T^*].
\end{split}
\ee
Inserting \eqref{a20}, \eqref{a19} and \eqref{a43} into \eqref{a16}, we end up with
\ben
\begin{split}
\|\na\tilde{c}(\cdot,t)\|_{L^r_x}\leq C_0(1+\|c_0\|_{W^{1,r}_x})+C_{T^*}\int_0^t (t-s)^{-1+\frac{1}{r}}\,ds
\leq C_{r,T^*}
\end{split}
\enn
for each $t\in (0,T^*]$ and $r\geq 2$, 
and thus
\be\label{a21}
\|\na c\|_{L^\infty_{T^*}L^r_x}\leq C_{r,T^*},\qquad \forall r\,\geq 2
\ee
thanks to the definition of $\tilde{c}$ and \eqref{a15}.
By the maximal Sobolev regularity (cf. \cite[3.1 Theorem]{matthias-jan1997}) applied to the heat semigroup with Neumann boundary conditions, Lemma \ref{l8}, \eqref{a14} and \eqref{a21}, one deduces for each $r\geq 2$ that
\be\label{a22}
\begin{split}
&\|\tilde{c}_t\|_{L^\rho_{T^*}L^{r}_x}+\|\tilde{c}\|_{L^\rho_{T^*}W^{2,r}_x}\\
\leq &C_0\|(u\cdot\na c+nc-\Delta g_2)e^{g_1}+|\na g_1|^2\tilde{c}-\Delta g_1 \tilde{c}-2\na g_1\cdot \na\tilde{c}\|_{L^\rho_{T^*}L^r_x}\\
\leq& C_0{(T^*)}^{1/\rho}(\|u\|_{L^\infty_{T^*}L^\alpha_x}\|\na c\|_{L^\infty_{T^*}L^{\frac{\alpha r}{\alpha-r}}_x}+\|n\|_{L^\infty_{T^*}L^\alpha_x}\|c\|_{L^\infty_{T^*}L^{\frac{\alpha r}{\alpha-r}}_x})\\
&+C_0{(T^*)}^{1/\rho}(\|\na c\|_{L^\infty_{T^*}L^r_x}+\|c\|_{L^\infty_{T^*}L^r_x}+1)\\
\leq &C_0{(T^*)}^{1/\rho}(\|u\|_{L^\infty_{T^*}H^1_x}\|\na c\|_{L^\infty_{T^*}L^{\frac{\alpha r}{\alpha-r}}_x}+\|n\|_{L^\infty_{T^*}L^\alpha_x}\|c\|_{L^\infty_{T^*}H^1_x})\\
&+C_0{(T^*)}^{1/\rho}(\|\na c\|_{L^\infty_{T^*}L^r_x}+\|c\|_{L^\infty_{T^*}H^1_x})\\
\leq &C_{r,\rho,T^*},
\end{split}
\ee
where $C_{r,\rho,T^*}$ is a constant depending on $r, \rho$ and $T^*$, and $\alpha$ is chosen to be larger than $r$. Collecting \eqref{a22}, \eqref{a21} and \eqref{a14} and using the definition of $\tilde{c}$ and \eqref{a15}, one derives the desired estimates. The proof is completed. 

\end{proof}
\begin{lemma}\label{l4} Suppose that the assumptions in Lemma \ref{l8} hold true. Then there exists a constant $C_{T^*}$ depending on $T^*$, such that
\ben
\|n\|_{L^\infty_{T^*}L^\infty_x}+\|\na n\|_{L^\infty_{T^*}L^2_x}+\|\Delta n\|_{L^2_{T^*}L^2_x}\leq C_{T^*}.
\enn
\end{lemma}
\begin{proof}
Define
\ben
\tilde{n}(x,t)=n(x,t)e^{-c(x,t)}.
\enn
Then the first and second equations and the boundary conditions of \eqref{cns} entail that
\be\label{a23}
\left\{
\begin{array}{lll}
\tilde{n}_t=\Delta \tilde{n}+\na\cdot(\tilde{n}\na c-u\tilde{n})+nc\tilde{n}-2\tilde{n}\Delta c,\quad   x\in \Om,\ t>0,\\
\tilde{n}(x,0)=\tilde{n}_0:=n_0e^{-c_0},\quad x\in \Omega,\\
\frac{\partial \tilde{n}}{\partial \nu}=0,\quad x\in \Gamma,\ \ t>0
\end{array}
\right.
\ee
and thus
\ben
\begin{split}
\tilde{n}(\cdot,t)=&e^{t\Delta}\tilde{n}_0+\int_0^t e^{(t-s)\Delta}\nabla\cdot (\tilde{n}\na c-u\tilde{n})(\cdot,s)\,ds
+\int_0^t e^{(t-s)\Delta}(nc\tilde{n}-2\tilde{n}\Delta c)(\cdot,s)\,ds\\
=:&\,J_1+J_2+J_3
\end{split}
\enn
for each $t\in (0,T^*]$, where $(e^{t\Delta})_{t\geq 0}$ is the Neumman heat semigroup. Let $B$ be the realization of $-\Delta+1$ with homogeneous Neumann boundary condition in $L^3(\Omega)$ and $\beta\in (\frac{1}{3},\frac{1}{2})$. Then $D(B^\beta)\hookrightarrow C^0(\bar{\Omega})$ and it follows from Lemma 1.3(iv) in \cite{winkler2010}, Lemma \ref{l8} and Lemma \ref{l3} that
\ben
\begin{split}
\|J_2\|_{L^\infty_x}\leq& C_0\int_0^t \|B^{\beta}e^{-(t-s)(B-1)}\nabla\cdot (\tilde{n}\na c-u\tilde{n})(\cdot,s)\|_{L^3_x}\,ds\\
\leq &C_0\int_0^t(t-s)^{-\frac{1}{2}-\beta}\|(\tilde{n}\na c-u\tilde{n})(\cdot,s)\|_{L^3_x}\,ds\\
\leq &C_0\int_0^t(t-s)^{-\frac{1}{2}-\beta}\|n(\cdot,s)\|_{L^{12}_x}(\|\na c(\cdot,s)\|_{L^4_x}+\|u(\cdot,s)\|_{L^4_x})\,ds\\
\leq &C_{0}\|n(\cdot,s)\|_{L^\infty_{T^*}L^{12}_x}(\|\na c\|_{L^\infty_{T^*}L^4_x}+\|u\|_{L^\infty_{T^*}L^4_x})\int_0^t(t-s)^{-\frac{1}{2}-\beta}\,ds\\
\leq& C_{T^*}(1+(T^*)^{\frac{1}{2}-\beta}).
\end{split}
\enn
Lemma 1.3(ii) in \cite{winkler2010}, the fact $c(x,t)\geq 0$ in $\Omega\times(0,T^*)$, \eqref{a36} and Lemma \ref{l3} yield
\ben
\begin{split}
\|J_3\|_{L^\infty_x}\leq C_0\|J_3\|_{W^{1,3}_x}
\leq&C_0 \int_0^t [\|\na e^{(t-s)\Delta}(nc\tilde{n}-2\tilde{n}\Delta c)(\cdot,s)\|_{L^3_x}+\|e^{(t-s)\Delta}(nc\tilde{n}-2\tilde{n}\Delta c)(\cdot,s)\|_{L^3_x}]\,ds\\
\leq &C_0\int_0^t[1+(t-s)^{-\frac{1}{2}}](\|n\|_{L^{4}_x}\|c\|_{L^{\infty}_x}+\|\Delta c\|_{L^{4}_x})\|n\|_{L^{12}_x}\,ds\\
\leq &C_0\|n\|_{L^\infty_{T^*}L^{4}_x}\|c\|_{L^\infty_{T^*}L^{\infty}_x}\|n\|_{L^\infty_{T^*}L^{12}_x}
\int_0^t[1+(t-s)^{-\frac{1}{2}}]\,ds\\
&+C_0\|\Delta c\|_{L^4_{T^*}L^{4}_x}\|n\|_{L^\infty_{T^*}L^{12}_x}\big\{\int_0^t[1+(t-s)^{-\frac{1}{2}}]^{\frac{4}{3}}\,ds\big\}^{\frac{3}{4}}\,ds\\
\leq &C_{T^*}[1+(T^*)^{\frac{1}{2}}]+C_{T^*}[1+(T^*)^{\frac{1}{3}}]^{\frac{3}{4}}.
\end{split}
\enn
Collecting the above estimates for $J_2$, $J_3$ and using the fact $\|J_1\|_{L^\infty_x}\leq \|\tilde{n}_0\|_{L^\infty_x}\leq \|n_0\|_{L^\infty_x}\leq C_0$ and \eqref{a36}, we conclude that
\be\label{a24}
\|n\|_{L^\infty_{T^*}L^\infty_x}\leq C_{T^*}.
\ee 
Taking the $L^2_{x}$ inner product of the equation in \eqref{a23} with $-\Delta \tilde{n}$ and using integration by parts, Lemma \ref{l8}, \eqref{a24} and \eqref{a36}, one gets
\ben
\begin{split}
\frac{1}{2}\frac{d}{dt}\|\na \tilde{n}\|_{L^2_x}^2+\|\Delta \tilde{n}\|_{L^2_x}^2
=&-\int_{\Omega}\na \tilde{n}\cdot\na c\Delta\tilde{n}\,dx
+\int_{\Omega}u\cdot\na \tilde{n}\Delta\tilde{n}\,dx+\int_{\Omega}\tilde{n}\Delta c\Delta\tilde{n}\,dx
-\int_{\Omega}nc \tilde{n}\Delta\tilde{n}\,dx\\
\leq &\|\na\tilde{n}\|_{L^4_x}(\|\na c\|_{L^4_{x}}+\|u\|_{L^4_{x}})\|\Delta\tilde{n}\|_{L^2_x}+\|\tilde{n}\|_{L^\infty_{x}}(\|\Delta c\|_{L^2_{x}}+\|n\|_{L^2_{x}}\|c\|_{L^\infty_{x}})\|\Delta \tilde{n}\|_{L^2_{x}}\\
\leq &C_0\|\na\tilde{n}\|_{L^2_x}^{\frac{1}{2}}(\|\na c\|_{L^4_{x}}+\|u\|_{H^1_{x}})\|\Delta\tilde{n}\|_{L^2_x}^{\frac{3}{2}}
+C_{T^*}(\|\Delta c\|_{L^2_{x}}+1)\|\Delta \tilde{n}\|_{L^2_{x}}\\
\leq &\frac{1}{2}\|\Delta \tilde{n}\|_{L^2_{x}}^2+C_0(\|\na c\|_{L^4_{x}}^4+\|u\|_{H^1_{x}}^4)\|\na\tilde{n}\|_{L^2_x}^2+C_{T^*}(\|\Delta c\|_{L^2_{x}}^2+1),
\end{split}
\enn
which, along with the Gronwall's inequality, Lemma \ref{l8} and \eqref{a13} gives
\ben
\|\na \tilde{n}\|_{L^\infty_{T^*}L^2_x}^2+\|\Delta \tilde{n}\|_{L^2_{T^*}L^2_x}^2\leq C_{T^*}
\enn
and thus 
\be\label{a25}
\|\na n\|_{L^\infty_{T^*}L^2_x}^2+\|\Delta n\|_{L^2_{T^*}L^2_x}^2\leq C_{T^*}
\ee
thanks to the definition of $\tilde{n}$, \eqref{a36}, Lemma \ref{l8} and \eqref{a24}. A combination of \eqref{a24} and \eqref{a25} completes the proof.

\end{proof}
\begin{lemma}\label{l5}
Let the assumptions in Lemma \ref{l8} hold true. Then there exists a constant $C_{T^*}$ depending on $T^*$, such that
\ben
\|u\|_{L^\infty_{T^*}L^\infty_x}+\|\na u\|_{L^\infty_{T^*}L^4_x}+\|u_t\|_{L^\infty_{T^*}L^2_x}\leq C_{T^*}.
\enn
\end{lemma}
\begin{proof}
Testing \eqref{a26} with $|\omega|^2\omega$ in $L^2_x$ and using the Gagliardo-Nirenberg inequality, we obtain
\ben
\begin{split}
\frac{1}{4}\frac{d}{dt}\|\omega\|_{L^4_x}^4=&\int_{\Omega}(\na\times n)\cdot(\na^{\perp}\phi) |\omega|^{2}\omega\,dx\\
\leq& \|\omega\|_{L^4_x}^4+C_0\|\na n\|_{L^4_x}^4\|\na\phi\|_{L^\infty_x}^4\\
\leq &\|\omega\|_{L^4_x}^4+C_0\|\na n\|_{L^2_x}^{2}\|\na^2 n\|_{L^2_x}^2
\end{split}
\enn
for each $t\in (0,T^*]$. 
Applying Gronwall's inequality to the above inequality and using Lemma \ref{l4}, one gets
\ben
\|\omega\|_{L^\infty_{T^*}L^4_x}\leq C_{T^*},
\enn
which, in conjunction with the Calderón-Zygmund inequality (see e.g. \cite[Proposition 7.5]{Bahouri-Chemin-Dachin2011}), Sobolev embedding inequality and Lemma \ref{l8} leads to
\be\label{a27}
\begin{split}
\|\na u\|_{L^\infty_{T^*}L^4_x}\leq C_0\|\omega\|_{L^\infty_{T^*}L^4_x}\leq C_{T^*},\qquad \|u\|_{L^\infty_{T^*}L^\infty_x}\leq C_0\| u\|_{L^\infty_{T^*}W^{1,4}_x}\leq C_{T^*}.
\end{split}
\ee
Taking the $L^2_x$ inner product of the third equation of \eqref{cns} with $u_t$ and using integration by parts to have
\ben
\begin{split}
\int_{\Omega}u^2_t\,dx=&-\int_{\Omega}(u\cdot\na u)\cdot u_t\,dx-\int_{\Omega}n\na \phi\cdot u_t\,dx\\
\leq &C_0\|u\|_{L^\infty_x}\|\na u\|_{L^2_x}\|u_t\|_{L^2_x}+C_0\|n\|_{L^2_x}\|\na \phi\|_{L^\infty_x}\|u_t\|_{L^2_x}\\
\leq &\frac{1}{2}\|u_t\|_{L^2_x}^2+C_0\|u\|_{L^\infty_x}^2\|\na u\|_{L^2_x}^2+C_0\|n\|_{L^2_x}^2.
\end{split}
\enn
Hence, it follows from \eqref{a27} and Lemma \ref{l8} that
\ben
\|u_t\|_{L^\infty_{T^*}L^2_x}\leq C_0\|u\|_{L^\infty_{T^*}L^\infty_x}\|\na u\|_{L^\infty_{T^*}L^2_x}^2+C_0\|n\|_{L^\infty_{T^*}L^2_x}^2\leq C_{T^*}, 
\enn
which, in conjunction with \eqref{a27} gives the desired estimates and the proof is finished.

\end{proof}
\begin{lemma}\label{l6} Let the assumptions in Lemma \ref{l8} hold true. Then there exists a constant $C_{T^*}$ depending on $T^*$, such that
\ben
\|n_t\|_{L^\infty_{T^*}L^2_x}+\|c_t\|_{L^\infty_{T^*}L^2_x}+\|\na n_t\|_{L^2_{T^*}L^2_x}+\|\na c_t\|_{L^2_{T^*}L^2_x}\leq C_{T^*}
\enn
and
\ben
\|n\|_{L^\infty_{T^*}H^2_x}+\|c\|_{L^\infty_{T^*}H^2_x}+\|n\|_{L^2_{T^*}H^3_x}+\|c\|_{L^2_{T^*}H^3_x}\leq C_{T^*}.
\enn
\end{lemma}
\begin{proof}
Differentiating the first equation of \eqref{cns} with respect to $t$ and testing the resulting equation with $n_t$ in $L^2_x$ and using integration by parts and Lemma \ref{l4}, one gets
\ben
\begin{split}
\frac{1}{2}\frac{d}{dt}&\|n_t\|_{L^2_x}^2+\|\na n_t\|_{L^2_x}^2\\
=&-\int_{\Omega} u_t\cdot\na n n_t\,dx
+\int_{\Omega}(n\na c_t+n_t \na c)\cdot\na n_t\,dx\\
\leq &\|u_t\|_{L^2_x}\|\na n\|_{L^4_x}\|n_t\|_{L^4_x}+(\|n\|_{L^\infty_x}\|\na c_t\|_{L^2_x}+\|n_t\|_{L^4_x}\|\na c\|_{L^4_x})\|\na n_t\|_{L^2_x}\\
\leq &C_0\|u_t\|_{L^2_x}\|\na n\|_{H^1_x}\|n_t\|_{L^2_x}^{\frac{1}{2}}\|\na n_t\|_{L^2_x}^{\frac{1}{2}}
+C_{T^*}(\|\na c_t\|_{L^2_x}+\|n_t\|_{L^2_x}^{\frac{1}{2}}\|\na n_t\|_{L^2_x}^{\frac{1}{2}}\|\na c\|_{H^1_x})\|\na n_t\|_{L^2_x}\\
\leq &\frac{1}{2}\|\na n_t\|_{L^2_x}^2+\frac{C_1}{2}\|\na c_t\|_{L^2_x}^2+C_{T^*}(\|\na c\|_{H^1_x}^4+1)\|n_t\|_{L^2_x}^2+C_0\|u_t\|_{L^2_x}^2\|\na n\|_{H^1_x}^2,
\end{split}
\enn
which, gives rise to
\be\label{a28}
\frac{d}{dt}\|n_t\|_{L^2_x}^2+\|\na n_t\|_{L^2_x}^2\leq C_1\|\na c_t\|_{L^2_x}^2+C_{T^*}(\|\na c\|_{H^1_x}^4+1)\|n_t\|_{L^2_x}^2+C_0\|u_t\|_{L^2_x}^2\|\na n\|_{H^1_x}^2,
\ee
where the constants $C_1$ and $C_{T^*}$ depend on $T^*$. 
Differentiating the second equation of \eqref{cns} with respect to $t$ and taking the $L^2_x$ inner product of the resulting equation with $c_t$ to have
\ben
\begin{split}
\frac{1}{2}\frac{d}{dt}&\|c_t\|_{L^2_x}^2+\|\na c_t\|_{L^2_x}^2+\|\sqrt{n}c_t\|_{L^2_x}^2\\
=&-\int_{\Omega} (u_t\cdot\na c+cn_t) c_t\,dx
+\int_{\Gamma}(\na c_t\cdot\nu)c_t\,dS\\
\leq &\|u_t\|_{L^2_x}\|\na c\|_{L^4_x}\|c_t\|_{L^4_x}+\|c\|_{L^\infty_x}\|n_t\|_{L^2_x}\|c_t\|_{L^2_x}
+\int_{\Gamma}[\kappa(\gamma-c)]_t c_t\,dS\\
\leq &C_0\|u_t\|_{L^2_x}\|\na c\|_{H^1_x}(\|c_t\|_{L^2_x}+\|\na c_t\|_{L^2_x})+C_0\|n_t\|_{L^2_x}\|c_t\|_{L^2_x}
-\int_{\Gamma}\kappa c_t^2\,dS\\
\leq &\frac{1}{2}\|\na c_t\|_{L^2_x}^2+C_0(\|u_t\|_{L^2_x}^2\|\na c\|_{H^1_x}^2+\|c_t\|_{L^2_x}^2+\|n_t\|_{L^2_x}^2)
-\int_{\Gamma}\kappa c_t^2\,dS,
\end{split}
\enn
where in the second inequality, we used the Gagliardo-Nirenberg inequality and \eqref{a36}.
Hence
\ben
\begin{split}
\frac{d}{dt}\|c_t\|_{L^2_x}^2+\|\na c_t\|_{L^2_x}^2+\|\sqrt{n}c_t\|_{L^2_x}^2+\int_{\Gamma}\kappa c_t^2\,dS
\leq C_0(\|u_t\|_{L^2_x}^2\|\na c\|_{H^1_x}^2+\|c_t\|_{L^2_x}^2+\|n_t\|_{L^2_x}^2).
\end{split}
\enn
Multiplying the above inequality with $(C_1+1)$ and adding the resulting inequality to \eqref{a28}, one gets
\ben
\begin{split}
&\frac{d}{dt}(\|n_t\|_{L^2_x}^2+\|c_t\|_{L^2_x}^2)+\|\na n_t\|_{L^2_x}^2+\|\na c_t\|_{L^2_x}^2+\|\sqrt{n}c_t\|_{L^2_x}^2+\int_{\Gamma}\kappa c_t^2\,dS\\
\leq &C_{T^*}(\|\na c\|_{H^1_x}^4+1)(\|n_t\|_{L^2_x}^2+\|c_t\|_{L^2_x}^2)+C_0\|u_t\|_{L^2_x}^2\|\na c\|_{H^1_x}^2,
\end{split}
\enn
which, along with Gronwall's inequality, Lemma \ref{l3}, Lemma \ref{l4} and Lemma \ref{l5}, gives rise to
\be\label{a29}
\|n_t\|_{L^\infty_{T^*}L^2_x}+\|c_t\|_{L^\infty_{T^*}L^2_x}+\|\na n_t\|_{L^2_{T^*}L^2_x}+\|\na c_t\|_{L^2_{T^*}L^2_x}\leq C_{T^*}.
\ee
It follows from the second equation of \eqref{cns}, \eqref{a29}, Lemma \ref{l3}, Lemma \ref{l4} and Lemma \ref{l5} that
\be\label{a30}
\begin{split}
&\|c\|_{L^2_{T^*}H^3_x}\\
\leq& \|c_t\|_{L^2_{T^*}H^1_x}+\|u\cdot\na c\|_{L^2_{T^*}H^1_x}+\|n c\|_{L^2_{T^*}H^1_x}\\
\leq &\|c_t\|_{L^2_{T^*}H^1_x}+ C_0(\|u\|_{L^\infty_{T^*}L^\infty_x}\|c\|_{L^2_{T^*}H^2_x}+\|\na u\|_{L^\infty_{T^*}L^4_x}\|\na c\|_{L^2_{T^*}L^4_x}+\|n\|_{L^\infty_{T^*}H^1_x} \|c\|_{L^2_{T^*}H^2_x})\\
\leq&C_{T^*}.
\end{split}
\ee
By a similar argument used above, one deduces that
\be\label{a31}
\|c\|_{L^\infty_{T^*}H^2_x}\leq \|c_t\|_{L^\infty_{T^*}L^2_x}+\|u\cdot\na c\|_{L^\infty_{T^*}L^2_x}+\|n c\|_{L^\infty_{T^*}L^2_x}\leq C_{T^*}.
\ee
The first equation of \eqref{cns} and the Gagliardo-Nirenberg inequality entail that
\ben
\begin{split}
\|n\|_{H^2_x}\leq& \|n_t\|_{L^2_x}+\|u\na n\|_{L^2_x}+\|n\na c\|_{H^1_x}\\
\leq & \|n_t\|_{L^2_x}+\|u\|_{L^\infty_x}\|\na n\|_{L^2_x}
+(\|n\|_{L^\infty_x}\|\na^2 c\|_{L^2_x}+\|\na n\|_{L^4_x}\|\na c\|_{L^4_x})\\
\leq &\|n_t\|_{L^2_x}+\|u\|_{L^\infty_x}\|\na n\|_{L^2_x}
+C_0(\|n\|_{L^{\infty}_x}\|c\|_{H^2_x}+\|n\|_{H^1_x}^{\frac{1}{2}}\|n\|_{H^2_x}^{\frac{1}{2}}
\|c\|_{H^1_x}^{\frac{1}{2}}\|c\|_{H^2_x}^{\frac{1}{2}})\\
\leq &\frac{1}{2}\|n\|_{H^2_x}+\|n_t\|_{L^2_x}+\|u\|_{L^\infty_x}\|\na n\|_{L^2_x}
+C_0(\|n\|_{L^\infty_x}\|c\|_{H^2_x}+\|n\|_{H^1_x}\|c\|_{H^1_x}\|c\|_{H^2_x})
\end{split}
\enn
for each $t\in (0,T^*]$, 
which, in conjunction with \eqref{a29}, \eqref{a31}, Lemma \ref{l3}, Lemma \ref{l4} and Lemma \ref{l5} leads to
\be\label{a32}
\begin{split}
\|n\|_{L^\infty_{T^*}H^2_x}
\leq& 2\|n_t\|_{L^\infty_{T^*}L^2_x}+2\|u\|_{L^\infty_{T^*}L^\infty_x}\|\na n\|_{L^\infty_{T^*}L^2_x}\\
&+C_0(\|n\|_{L^\infty_{T^*}L^\infty_x}\|c\|_{L^\infty_{T^*}H^2_x}+\|n\|_{L^\infty_{T^*}H^1_x}\|c\|_{L^\infty_{T^*}H^2_x}^2)\leq C_{T^*}.
\end{split}
\ee
In a similar fashion as above, we obtain
\ben
\|n\|_{L^2_{T^*}H^3_x}\leq \|n_t\|_{L^2_{T^*}H^1_x}+\|u\na n\|_{L^2_{T^*}H^1_x}+\|n\na c\|_{L^2_{T^*}H^2_x}\leq C_{T^*},
\enn
which, along with \eqref{a29}, \eqref{a30}, \eqref{a31} and \eqref{a32} completes the proof.

\end{proof}
\begin{lemma}\label{l7}
Suppose that the assumptions in Lemma \ref{l8} hold true. Then there exists a constant $C_{T^*}$ depending on $T^*$, such that
\ben
\|u\|_{L^\infty_{T^*}H^3_x}+\|\na p\|_{L^\infty_{T^*}H^2_x}+\|\na p\|_{L^2_{T^*}H^3_x}\leq C_{T^*}.
\enn
\end{lemma}
\begin{proof}
Let $X(\cdot,t):\Omega\rightarrow \Omega$ be the particle-trajectory map of the fluid (see e.g. \cite{majda-bertozzi2002}), which is defined by
\ben
\left\{
\begin{array}{lll}
\frac{d}{dt}X(y,t)=u(X(y,t),t),\qquad y\in\Omega,\ \ t\in(0,T^*],\\
X(y,0)=y,\qquad y\in\Omega.
\end{array}
\right.
\enn
Then it follows from \eqref{a26} that
\ben
\omega(X(y,t),t)=\omega_0(y,0)+\int_0^t
(\na\times n(X(y,s),s)) \cdot(\na^{\perp} \phi(X(y,s)))\,ds
\enn
for each $t\in (0,T^*]$, and therefore, the Sobolev embedding inequality and Lemma \ref{l6} imply that 
\be\label{a34}
\begin{split}
\|\omega\|_{L^\infty_{T^*}L^\infty_x}\leq& \|\omega_0\|_{L^\infty_x}
+\|\na \phi\|_{L^\infty_x}\int_0^{T^*}\|\na n(\cdot,s)\|_{L^\infty_x}\,ds\\
\leq& C_0\|u_0\|_{H^3_x}
+C_0\|\na \phi\|_{L^\infty_x}(T^*)^{1/2}\|n\|_{L^2_{T^*}H^3_x}
\leq C_{T^*}.
\end{split}
\ee
Applying $D^{\alpha}_x$ with $0\leq |\alpha|\leq 3$ to the third equation of \eqref{cns} and taking the $L^2_x$ inner product of the resulting equation with $D^{\alpha}_x u$ and using the fact $\na\cdot u=0$ to have
\be\label{a33}
\begin{split}
\frac{1}{2}\frac{d}{dt}\|D^{\alpha}_x u\|_{L^2_x}^2
=&-\int_{\Omega}[D^{\alpha}_x(u\cdot \na u)-u\cdot (D^{\alpha}_x\na u)]\cdot (D^{\alpha}_xu)\,dx\\
&-\int_{\Omega}[D^{\alpha}_x\na p\!+\!D^{\alpha}_x(n\na \phi)]\cdot(D^{\alpha}_x u)\,dx\\
=:&\,K_1+K_2.
\end{split}
\ee
Recalling the following commutator estimate (see e.g. \cite[Lemma 1]{ferrari1993}):
\ben
\|D^{\alpha}_x(fg)-fD^{\alpha}_x g\|_{L^2_x}\leq C_0(\|f\|_{H^m_x}\|g\|_{L^\infty_x}+\|f\|_{W^{1,\infty}_x}\|g\|_{H^{m-1}_x})
\enn
for $0\leq |\alpha|\leq m$ with $m\geq 3$ and $f\in H^{m}(\Omega)\cap C^{1}(\Omega)$ and $g\in H^{m-1}(\Omega)\cap C(\Omega)$, one gets
\ben
\begin{split}
K_1\leq& \|D^{\alpha}_x(u\cdot\na u)-u\cdot (D^{\alpha}_x\na u)\|_{L^2_x}\|D^{\alpha}_xu\|_{L^2_x}\\
\leq& C_0(\|u\|_{H^3_x}\|\na u\|_{L^\infty_x}+\|u\|_{W^{1,\infty}_x}\|\na u\|_{H^{2}_x})\|u\|_{H^3_x}\\
\leq& C_0\|u\|_{W^{1,\infty}_x}\|u\|_{H^3_x}^2.
\end{split}
\enn
By the classical estimates on $\na p$ involving $u$ from the Euler equation (see e.g. \cite[Lemma 1.2]{temam1975}, \cite[Lemma 2]{ferrari1993}), one deduces that
\be\label{a35}
\|\na p\|_{H^3_x}\leq C_0(\|u\|_{W^{1,\infty}_x}\|u\|_{H^3_x}+\|n\na\phi\|_{H^3_x})\leq C_0(\|u\|_{W^{1,\infty}_x}\|u\|_{H^3_x}+\|n\|_{H^3_x}).
\ee
Thus
\ben
\begin{split}
K_2\leq (\|\na p\|_{H^3_x}+\|n\na\phi\|_{H^3_x})\|u\|_{H^3_x}\leq C_0(1+\|u\|_{W^{1,\infty}_x})\|u\|_{H^3_x}^2+C_0\|n\|_{H^3_x}^2.
\end{split}
\enn
Inserting the above estimates for $K_1$ and $K_2$ into \eqref{a33} and summing the resulting equations with respect to $0\leq |\alpha|\leq 3$ to get
\ben
\frac{d}{dt}\|u\|_{H^3_x}^2\leq C_0(1+\|u\|_{W^{1,\infty}_x})\|u\|_{H^3_x}^2+C_0\|n\|_{H^3_x}^2,
\enn
which, in conjunction with the following Brezis-Gallouet-Wainger type inequality (see e.g. \cite[Proposition 1]{ferrari1993}):
\ben
\|u\|_{W^{1,\infty}_x}\leq C_0[(1+\log^{+}\|u\|_{H^3_x})\|\omega\|_{L^\infty_x}+1]
\enn
and \eqref{a34} gives rise to
\ben
\begin{split}
\frac{d}{dt}\|u\|_{H^3_x}^2\leq C_{T^*}(1+\log^{+}\|u\|_{H^3_x})\|u\|_{H^3_x}^2+C_0\|n\|_{H^3_x}^2.
\end{split}
\enn
Applying the Gronwall's inequality to the above inequality and using Lemma \ref{l6}, we obtain 
\be\label{a40}
\|u\|_{L^\infty_{T^*}H^3_x}\leq C_{T^*},
\ee
 which, along with the Sobolev embedding inequality, \eqref{a35} and Lemma \ref{l6} yields
\be\label{a41}
\begin{split}
\|\na p\|_{L^2_{T^*}H^3_x}\leq& C_0(T^*)^{\frac{1}{2}}\|u\|_{L^\infty_{T^*}W^{1,\infty}_x}\|u\|_{L^\infty_{T^*}H^3_x}+C_0\|n\|_{L^2_{T^*}H^3_x}\\
\leq & C_0(T^*)^{\frac{1}{2}}\|u\|_{L^\infty_{T^*}H^3_x}^2+C_0\|n\|_{L^2_{T^*}H^3_x}\\
\leq &C_{T^*}.
\end{split}
\ee
Applying $\na\cdot$ to the third equation of \eqref{cns} and using the fact $\nabla\cdot u=0$ and the third boundary condition in \eqref{bc}, one gets
\ben
\left\{
\begin{array}{lll}
\Delta p=-\sum\limits_{i,j=1}^2\partial_i u_j\partial_j u_i-\na\cdot(n\na \phi),\quad   x\in \Om,\\
\frac{\partial p}{\partial \nu}=\sum\limits_{i,j=1}^2u_iu_j \sigma_{ij}-n\na\phi\cdot\nu,\quad x\in \Gamma
\end{array}
\right.
\enn
 for fixed $t\in (0,T^*]$, where $\sigma_{ij}\in H^\infty(\Gamma)$ denotes the normalized Hessian of the local boundary function $\varphi$, i.e. $\sigma_{ij}=\frac{\partial_i\partial_j\varphi}{|\na \varphi|}$,
 representing the curvature contribution of $\Gamma$ (cf. \cite[Lemma 1.1]{temam1975}). Employing the standard theory on elliptic systems (see e.g. \cite[Theorem 5.4]{lions&magenes}) and the trace theorem, we deduce from the above system that
 \be\label{a44}
 \begin{split}
 \|\na p\|_{H^2(\Omega)}\leq & C_0[\sum\limits_{i,j=1}^2\|\partial_i u_j\partial_j u_i\|_{H^1(\Omega)}+\|n\na \phi\|_{H^2(\Omega)}]\\
 &+C_0[\sum\limits_{i,j=1}^2\|u_iu_j\|_{H^{3/2}(\Gamma)}+\|n\na \phi\|_{H^{3/2}(\Gamma)}]\\
 \leq &C_0\|u\|_{H^3(\Omega)}^2+C_0\|n\|_{H^2(\Omega)}
 \end{split}
 \ee
 for each $t\in (0,T^*]$, which, in conjunction with \eqref{a40} and Lemma \ref{l6} leads to
 \be\label{a42}
 \|\na p\|_{L^\infty_{T^*}H^2_x}\leq C_0 \|u\|_{L^\infty_{T^*}H^3_x}^2+C_0\|n\|_{L^\infty_{T^*}H^2_x}\leq C_{T^*}.
 \ee
 Collecting \eqref{a40}, \eqref{a41} and \eqref{a42}, one gets the desired estimates. The proof is completed.
 
\end{proof}

With Lemma \ref{l8}- Lemma \ref{l7} in hand, we are now in a position to prove Theorem \ref{t1}.
~\\
\textbf{Proof of Theorem \ref{t1}.} Under the assumption $T^*<\infty$, we next justify \eqref{a00} by the argument of contradiction. Indeed, if \ben
\|n\|_{L^q(0,T^*;L^p(\Omega))}<\infty,\qquad \frac{2}{p}+\frac{2}{q}\leq 1
\enn
for some $2<p\leq \infty$. Then it follows from Lemma \ref{l8}-Lemma \ref{l7} that
\be\label{b12}
\begin{split}
(n,c,u,\na p)\in C([0,T^{*}];H^2_x\times H^2_x\times H^3_x\times H^2_x)
\end{split}
\ee
with
\ben
\|n\|_{L^\infty_{T^*}H^2_x}+\|c\|_{L^\infty_{T^*}H^2_x}+\|u\|_{L^\infty_{T^*}H^3_x}+\|\na p\|_{L^\infty_{T^*}H^2_x}\leq C_{T^*},
\enn
for some constant $C_{T^*}<\infty$ depending on $T^*$. In particular, 
\ben
\|n(\cdot,T^*)\|_{H^2_x}+\|c(\cdot,T^*)\|_{H^2_x}+\|u(\cdot,T^*)\|_{H^3_x}\leq C_{T^*}<\infty,
\enn
from which, and Proposition \ref{p1}  we conclude that there is a $\delta>0$ depending on $\|n(\cdot,T^*)\|_{H^2_x}$, $\|c(\cdot,T^*)\|_{H^2_x}$ and $\|u(\cdot,T^*)\|_{H^3_x}$, such that
\ben
(n,c,u,\na p)\in C([T^*,T^{*}+\delta];H^2_x\times H^2_x\times H^3_x\times H^2_x),
\enn
which, in conjunction with \eqref{b12} leads to
\ben
(n,c,u,\na p)\in C([0,T^{*}+\delta];H^2_x\times H^2_x\times H^3_x\times H^2_x).
\enn
This contradicts with the definition of $T^*$ and thus \eqref{a00} holds true. The proof is finished.

\endProof

\section{Proof of Theorem \ref{t2}}
 We argue by contradiction. We suppose that $T^*$, the maximal time of existence in Proposition \ref{p1} is finite. We shall show that $\|n\|_{L^4_{T^*}L^4_x}<\infty$, which contradicts Theorem \ref{t1}.
First, it follows from the first inequality in \eqref{a36} and \eqref{a0} that
\be\label{a2}
\|c\|_{L^\infty_{T^*}L^\infty_x}\leq \max\{\|\gamma\|_{L^\infty(\Gamma)}, \|c_0\|_{L^\infty}\}\leq \frac{1}{48}.
\ee
Let 
\ben
\theta(c)=e^{12c^2}.
\enn
Then a direct calculation and \eqref{a2} lead to
\be\label{a4}
\theta^{'}\geq 0,\qquad \frac{5(\theta^{'})^2}{\theta}+\theta+\theta^{'}\leq \frac{\theta^{''}}{4},
\ee
thanks to the fact that $c(x,t)\geq 0$ in $\Omega\times [0,T^*]$. 
Multiplying the first equation of \eqref{cns} by $\theta(c) n$ in $L^2_x$ and using the second equation in \eqref{cns} and integration by parts to have
\be\label{a1}
\begin{split} 
\frac{1}{2}\frac{d}{dt}\int_{\Omega}\theta n^2dx
=&\int_{\Omega}\theta n[-u\cdot \na n+\Delta n-\na\cdot (n\na c)]\,dx\\
&+\frac{1}{2}\int_{\Omega}\theta^{'}n^2(-u\cdot \na c+\Delta c-nc)dx\\
=
&-\int_{\Omega}\theta |\na n|^2dx
-\frac{1}{2}\int_{\Omega}\theta^{''}n^2|\na c|^2dx
+\int_{\Omega}\theta n\na n\cdot\na c\, dx\\
&-2\int_{\Omega}\theta^{'} n\na n\cdot\na c \,dx
+\int_{\Omega}\theta^{'} n^2 |\na c|^2 dx
+\frac{1}{2}\int_{\Gamma}\theta^{'}n^2\na c\cdot \nu dS \\
&-\frac{1}{2}\int_{\Omega}\theta^{'}n^3 c \,dx
\end{split}
\ee
for each $t\in (0,T^*]$, where we used 
\ben
\int_{\Omega} \theta n u\cdot \na n \,dx+\frac{1}{2}\int_{\Omega} \theta^{'} n^2 u\cdot \na c\,dx
=\frac{1}{2}\int_{\Omega} \theta u\cdot \na n^2\,dx +\frac{1}{2}\int_{\Omega} n^2\na\cdot (\theta u)\,dx
=0,
\enn
thanks to the facts $\na\cdot u=0$ in $\Omega$ and $u\cdot\nu=0$ on $\Gamma$. The Cauchy-Schwarz inequality leads to
\be\label{a5}
\begin{split}
\int_{\Omega}\theta n\na n\cdot\na c\, dx
-2\int_{\Omega}\theta^{'} n\na n\cdot\na c \,dx
\leq \frac{1}{2}\int_{\Omega}\theta |\na n|^2dx
+\int_{\Omega}\big(\theta+\frac{4(\theta^{'})^2}{\theta}\big) n^2|\na c|^2dx.
\end{split}
\ee
It follows from the first boundary condition in \eqref{bc}, the trace theorem and \eqref{a2} that
\be\label{a6}
\begin{split}
\frac{1}{2}\int_{\Gamma}\theta^{'}n^2\na c\cdot \nu dS
=&\frac{1}{2}\int_{\Gamma}\theta^{'}n^2\kappa(x)(\gamma(x)-c)\,dS\\
\leq& \|\kappa(x)\gamma(x)\|_{L^\infty(\Gamma)}\big\|\frac{\theta^{'}}{\theta}\big\|_{L^\infty(\Gamma)}\|n\sqrt{\theta}\|_{L^2(\Gamma)}^2\\
\leq& C_0\|\kappa\|_{H^1(\Gamma)}\|\gamma\|_{H^1(\Gamma)}\|c\|_{L^\infty_x}\|n\sqrt{\theta}\|_{L^2_x}\|\na (n\sqrt{\theta})\|_{L^2_x}\\
\leq &\frac{1}{4}\|\sqrt{\theta}\na n\|_{L^2_x}^2+\int_{\Omega}\frac{(\theta^{'})^2}{\theta} n^2|\na c|^2dx+C_0\|n\sqrt{\theta}\|_{L^2_x}^2,
\end{split}
\ee
where in the second inequality we have used the fact that
\ben
\big\|\frac{\theta^{'}}{\theta}\big\|_{L^\infty_TL^\infty_x}=\|2\cdot 12^2 c\|_{L^\infty_TL^\infty_x}\leq C_0,
\enn
due to \eqref{a2}. Substituting \eqref{a5} and \eqref{a6} into \eqref{a1} and using \eqref{a4}, one gets
\ben
\frac{d}{dt}\int_{\Omega}\theta n^2dx+\frac{1}{2}\int_{\Omega}\theta |\na n|^2dx
+\frac{1}{2}\int_{\Omega}\theta^{''}n^2|\na c|^2dx\leq C_0\|n\sqrt{\theta}\|_{L^2_x}^2, 
\enn
which, along with Gronwall's inequality and the fact $\theta\geq 1$ entails that
\be\label{a3}
\|n\|_{L^\infty_{T^*}L^2_x}+\|\na n\|_{L^2_{T^*}L^2_x}\leq C_{T^*},
\ee
where the constant $C_{T^*}$ depends on $T^*$. Then the Gagliardo-Nirenberg inequality and \eqref{a3} yield
\ben
\|n\|_{L^4_{T^*}L^4_x}\leq C_0\|n\|_{L^\infty_{T^*}L^2_x}^{\frac{1}{2}}\|n\|_{L^2_{T^*}H^1_x}^{\frac{1}{2}}\leq C_{T^*}<\infty,
\enn
which, contradicts with Theorem \ref{t1} and thus $T^*=\infty$. 
The proof is completed.

\endProof

\section{Appendix}
In this section, we give a sketch proof of Proposition \ref{p1}. Enlightened by \cite{braukhoff2017} we first transform \eqref{cns}-\eqref{bc} into a  system with homogeneous boundary conditions. Then we employ the Banach's fixed-point theorem to prove local existence of solutions to this transformed system and finally we pass these results to the pre-transformed system and derive Proposition \ref{p1}.
~\\  
\textbf{The transformed system and its local well-posedness}. 
Let
\be\label{b10}
\tilde{c}(x,t)=e^{g_1(x)}(g_2(x)-c(x,t)),\qquad \tilde{n}(x,t)=n(x,t)e^{-c(x,t)},
\ee
where $g_1, g_2\in H^3(\Omega)$ satisfy \eqref{b01} and \eqref{a15}.
Then
\be\label{b2}
\begin{split}
&n_t+u\cdot\na n=e^c(\tilde{n}_t+u\cdot\na\tilde{n})+e^c\tilde{n}(c_t+u\cdot\na c)
=e^c(\tilde{n}_t+u\cdot\na\tilde{n})+e^c\tilde{n}(\Delta c-nc),\\
&\Delta n-\nabla\cdot(n\nabla c)=\Delta(e^c\tilde n )-\nabla\cdot(e^c\tilde{n}\nabla c)=e^c(\Delta \tilde{n}+\nabla\tilde{n}\cdot\nabla c),\\
&\nabla c=\nabla g_2+\tilde{c}e^{-g_1}\nabla g_1-e^{-g_1}\nabla\tilde{c},\\
&\Delta c=\Delta g_2+e^{-g_1}(\tilde{c}\Delta g_1-\tilde{c}|\nabla g_1|^2+2\nabla g_1\cdot\nabla \tilde{c}-\Delta\tilde{c}).
\end{split}
\ee
Substituting \eqref{b2} into \eqref{cns}-\eqref{bc}, one gets
\be\label{e01}
\left\{
\begin{array}{lll}
\tilde{n}_t+u\cdot\nabla \tilde{n}=\Delta \tilde{n}+e^{-g_1}(\tilde{n}\Delta \tilde{c}-\nabla\tilde{n}\cdot\nabla\tilde{c})\\
\qquad\qquad\ \ \ \ \ +(\nabla g_2+e^{-g_1}\tilde{c}\nabla g_1)\cdot\nabla\tilde{n}+\tilde{n}e^{-g_1}F,\quad\qquad   x\in \Om,\ t>0,\\
\tilde{c}_t+u\cdot\na \tilde{c}=\Delta \tilde{c}+u\cdot\nabla g_1\tilde{c}+e^{g_1}u\cdot\nabla g_2+F,\quad \quad\qquad\quad   x\in \Om,\ t>0,\\
u_t+u\cdot\nabla u+\nabla p=-\tilde{n}e^{g_2}e^{-\tilde{c}e^{-g_1}}\nabla\phi,\quad \quad\qquad \quad\qquad\ \    \,x\in \Om,\ t>0,\\
\na\cdot u=0,\quad  \quad \quad\qquad \quad\qquad\quad \quad\qquad \quad\qquad \qquad \quad \ \ x\in \Om,\ t>0,\\
\tilde{n}(x,0)=\tilde{n}_0:=n_0(x)e^{-c_0(x)},\,\,u(x,0)=u_0(x),\quad \qquad \qquad \quad \quad x\in \Omega,\\
\tilde{c}(x,0)=\tilde{c}_0:=e^{g_1(x)}(g_2(x)-c_0(x)),\quad \qquad \qquad \quad \qquad \quad \quad \ x\in \Omega,\\
\frac{\partial \tilde{n}}{\partial \nu}=\frac{\partial \tilde{c}}{\partial \nu}=u\cdot\nu=0,\quad\quad \qquad \qquad \quad \qquad 
\quad \qquad \qquad x\in \Gamma,\ \ t>0
\end{array}
\right.
\ee
with
\ben
\begin{split}
F:=&-2\nabla g_1\cdot\nabla\tilde{c}-e^{g_1}\Delta g_2+(|\na g_1|^2-\Delta g_1) \tilde{c}+e^{g_2-\tilde{c}e^{-g_1}}\tilde{n}[e^{g_1}g_2-\tilde{c}].
\end{split}
\enn

We next prove the local well-posedness of \eqref{e01} by an application of the Banach's fixed-point theorem. For $T>0$, denote
\ben
\begin{split}
X_{T}=\{(\tilde{n},\tilde{c},\tilde{u})\in L^1(\Omega\times(0,T)):\,\,&\tilde{n},\tilde{c}\in L^2(0,T;H^3),\tilde{n}_t,\tilde{c}_t\in L^2(0,T;H^1),u\in L^\infty(0,T;H^3),\\
&\nabla\cdot u=0\,\,\text{for}\,\,(x,t)\in\Omega\times(0,T],\,\,u\cdot\nu=0\,\,\text{for}\, \,(x,t)\in \Gamma\times(0,T]\}
\end{split}
\enn
and  
\ben
\|(\tilde{n},\tilde{c},u)\|_{X_{T}}=\|\tilde{n}\|_{L^2_TH^3_x}+\|\tilde{c}\|_{L^2_TH^3_x}
+\|\tilde{n}_t\|_{L^2_TH^1_x}
+\|\tilde{c}_t\|_{L^2_TH^1_x}+\|u\|_{L^\infty_TH^3_x}.
\enn
We define the mapping $\Phi:=(\Phi_1,\Phi_2,\Phi_3): X_T\rightarrow X_T$ as solution of the following linear problem
\be\label{e00}
\left\{
\begin{array}{lll}
\partial_t\Phi_1-\Delta \Phi_1=e^{-g_1}(\tilde{n}\Delta \Phi_2-\nabla\tilde{n}\cdot\nabla\Phi_2)-u\cdot\nabla \tilde{n}\\
\qquad\qquad\qquad+(\nabla g_2+e^{-g_1}\tilde{c}\nabla g_1)\cdot\nabla\tilde{n}+\tilde{n}e^{-g_1}F,\qquad   x\in \Om,\ t>0,\\
\partial_t\Phi_2-\Delta \Phi_2=-u\cdot\na \tilde{c}+u\cdot\nabla g_1\tilde{c}+e^{g_1}u\cdot\nabla g_2+F,\qquad   x\in \Om,\ t>0,\\
\partial_t\Phi_3+u\cdot\nabla \Phi_3+\nabla P=-\tilde{n}e^{g_2}e^{-\tilde{c}e^{-g_1}}\nabla\phi,\qquad   \qquad \quad \ \ x\in \Om,\ t>0,\\
\na\cdot \Phi_3=0,\quad  \qquad   \qquad\qquad   \qquad   \qquad\qquad\qquad\quad\quad \ \ x\in \Om,\ t>0,\\
(\Phi_1,\Phi_2,\Phi_3)(x,0)=(\tilde{c}_0,\tilde{n}_0,u_0)(x),\quad \qquad\qquad\qquad\quad\quad\quad\ \  x\in \Omega,\\
\frac{\partial \Phi_1}{\partial \nu}=\frac{\partial \Phi_2}{\partial \nu}=\Phi_3\cdot\nu=0,\quad\qquad\qquad\qquad\qquad\quad\quad\quad x\in \Gamma,\ \ t>0,
\end{array}
\right.
\ee
where $\nabla P$ is the pressure associated with $\Phi_3$. For given $(\tilde{n},\tilde{c},u)\in X_T$, the solvability of \eqref{e00} in $X_T$ follows from the well-posedness theory on linear parabolic systems and Euler equations, we refer the reader to \cite[Section 7.1]{evans} and \cite{temam1975}. We next show that the mapping $\Phi: X_T\rightarrow X_T$ is bounded. 
First it follows from the Sobolev embedding inequality and \eqref{a15} that
\be\label{b3}
\begin{split}
&\|F\|_{L^2_TH^1_x}\\
\leq& C_0T^{\frac{1}{2}}[\|g_1\|_{H^3_x}\|\tilde{c}\|_{L^\infty_TH^2_x}+(\|g_1\|_{H^2_x}+1)\|g_2\|_{H^3_x}\,e^{C_0\|g_1\|_{L^\infty_x}}+(\|g_1\|_{H^3_x}^2+\|g_1\|_{H^3_x})
\|\tilde{c}\|_{L^\infty_T H^2_x}]\\
&+C_0 T^{\frac{1}{2}}[e^{\|g_2\|_{L^\infty_x}}\exp\{\|\tilde{c}\|_{L^\infty_TL^\infty_x}\exp\{\|g_1\|_{L^\infty_x}\}\}(e^{\|g_1\|_{L^\infty_x}}\|g_2\|_{H^2_x}+\|\tilde{c}\|_{L^\infty_TH^2_x}^2+1)
\|\tilde{n}\|_{L^\infty_TH^2_x}]\\
\leq &C_0T^{\frac{1}{2}}(\|\tilde{c}\|_{L^\infty_TH^2_x}+1)
+C_0T^{\frac{1}{2}}e^{C_0\|\tilde{c}\|_{L^\infty_TH^2_x}}(\|\tilde{c}\|_{L^\infty_TH^2_x}^2+1)\|\tilde{n}\|_{L^\infty_TH^2_x},
\end{split}
\ee
where the constant $C_0$ depends only on $\Omega$ and the boundary data. Recall the following Sobolev embedding inequality involving time (see e.g. \cite[Section 5.9.2]{evans}):
\be\label{b4}
\|f\|_{L^\infty_TH^2_x}\leq C_0\|f(\cdot,0)\|_{H^2_x}+C_0\|f\|_{L^2_TH^3_x}^{\frac{1}{2}}\|f_t\|_{L^2_TH^1_x}^{\frac{1}{2}}
\ee
for each $f$ satisfying $f\in L^2(0,T;H^3),\, f_t\in L^2(0,T;H^1)$, where the constant $C_0$ depends only on $\Omega$. Then we derive from \eqref{b3} and \eqref{b4} that
\be\label{b5}
\|F\|_{L^2_TH^1_x}\leq C_0T^{\frac{1}{2}}(1+\|(\tilde{n},\tilde{c},u)\|_{X_T}^3)(1+e^{C_0(1+\|(\tilde{n},\tilde{c},u)\|_{X_T})}),
\ee
where the constant $C_0$ depends on $\Omega$ and the initial and boundary data. The maximal Sobolev regularity on heat semigroup subject to Neumann boundary conditions (see e.g. \cite[3.1 Theorem]{matthias-jan1997}), the Sobolev embedding inequality, \eqref{a15}, \eqref{b4} and \eqref{b5} lead to 
\be\label{b6}
\begin{split}
&\|\partial_t\Phi_2\|_{L^2_TH^1_x}+\|\Phi_2\|_{L^2_TH^3_x}\\
\leq &\|u\cdot\na \tilde{c}\|_{L^2_TH^1_x}+\|u\cdot\nabla g_1\tilde{c}\|_{L^2_TH^1_x}+\|e^{g_1}u\cdot\nabla g_2\|_{L^2_TH^1_x}+\|F\|_{L^2_TH^1_x}\\
\leq &C_0T^{\frac{1}{2}}[(\|g_1\|_{H^2_x}+1)\|u\|_{L^\infty_TH^2_x}\|\tilde{c}\|_{L^\infty_TH^2_x}
+e^{\|g_1\|_{L^\infty_x}}(\|g_1\|_{H^2_x}+1)\|g_2\|_{H^2_x}\|u\|_{L^\infty_TH^2_x}]+\|F\|_{L^2_TH^1_x}\\
\leq &C_0T^{\frac{1}{2}}[(\|g_1\|_{H^2_x}+1)\|u\|_{L^\infty_TH^2_x}\|\tilde{c}\|_{L^\infty_TH^2_x}
+e^{C_0\|g_1\|_{H^2_x}}(\|g_1\|_{H^2_x}+1)\|g_2\|_{H^2_x}\|u\|_{L^\infty_TH^2_x}]+\|F\|_{L^2_TH^1_x}\\
\leq& C_0T^{\frac{1}{2}}(1+\|(\tilde{n},\tilde{c},u)\|_{X_T}^3)(1+e^{C_0(1+\|(\tilde{n},\tilde{c},u)\|_{X_T})}).
\end{split}
\ee
A similar argument used in deriving \eqref{b6}, along with \eqref{b5} and \eqref{b6} entails that
\be\label{b7}
\begin{split}
&\|\partial_t\Phi_1\|_{L^2_TH^1_x}+\|\Phi_1\|_{L^2_TH^3_x}\\
\leq &\|u\cdot\nabla \tilde{n}\|_{L^2_TH^1_x}+\|(\nabla g_2+e^{-g_1}\tilde{c}\nabla g_1)\cdot\nabla\tilde{n}\|_{L^2_TH^1_x}\\
&+\|e^{-g_1}(\tilde{n}\Delta \Phi_2-\nabla\tilde{n}\cdot\nabla\Phi_2)\|_{L^2_TH^1_x}+\|\tilde{n}e^{-g_1}F\|_{L^2_TH^1_x}\\
\leq &C_0T^{\frac{1}{2}}(\|u\|_{L^\infty_TH^2_x}\|\tilde{n}\|_{L^\infty_TH^2_x}+\|g_2\|_{H^3_x}\|\tilde{n}\|_{L^\infty_TH^2_x})\\
&+C_0T^{\frac{1}{2}}e^{\|g_1\|_{L^\infty_x}}(\|g_1\|_{H^3_x}+1)\|g_1\|_{H^3_x}\|\tilde{c}\|_{L^\infty_TH^2_x}\|\tilde{n}\|_{L^\infty_TH^2_x}\\
&+e^{\|g_1\|_{L^\infty_x}}(\|g_1\|_{H^3_x}+1)\|\tilde{n}\|_{L^\infty_TH^2_x}(\|\Phi_2\|_{L^2_TH^3_x}+\|F\|_{L^2_TH^1_x})\\
\leq& C_0T^{\frac{1}{2}}(1+\|(\tilde{n},\tilde{c},u)\|_{X_T}^4)(1+e^{C_0(1+\|(\tilde{n},\tilde{c},u)\|_{X_T})}).
\end{split}
\ee
We proceed to estimating $\Phi_3$. For $0\leq|\alpha|\leq 3$, applying $D_x^\alpha$ to the third equation of \eqref{e00} and taking the $L^2_x$ inner product of the resulting results with $D_x^\alpha \Phi_3$ and using the facts $\na\cdot u=0$ in $\Omega\times(0,T]$ and $u\cdot\nu=0$ on $\Gamma\times(0,T]$, we have
\be\label{b8}
\begin{split}
\frac{1}{2}\frac{d}{dt}\|D^\alpha_x\Phi_3\|_{L^2_x}^2=&-\int_{\Omega}[D^{\alpha}_x(u\cdot \na \Phi_3)-u\cdot (D^{\alpha}_x\na \Phi_3)]\cdot D^{\alpha}_x\Phi_3\,dx\\
&-\int_{\Omega}[D^{\alpha}_x\na P+D^{\alpha}_x(\tilde{n}e^{g_2}e^{-\tilde{c}e^{-g_1}}\na \phi)]\cdot(D^{\alpha}_x \Phi_3)\,dx\\
=:&\,L_1+L_2.
\end{split}
\ee
Similarly to the derivation of \eqref{a44}, one gets
\ben
 \begin{split}
 \|\na P\|_{H^3_x}\leq & C_0[\sum\limits_{i,j=1}^2\|\partial_i u_j\partial_j \Phi_{3i}\|_{H^2_x}+\|\tilde{n}e^{g_2}e^{-\tilde{c}e^{-g_1}}\nabla\phi\|_{H^3_x}]\\
 &+C_0[\sum\limits_{i,j=1}^2\|u_i\Phi_{3j}\|_{H^{5/2}(\Gamma)}+\|\tilde{n}e^{g_2}e^{-\tilde{c}e^{-g_1}}\nabla\phi\|_{H^{5/2}(\Gamma)}]\\
 \leq &C_0\|u\|_{H^3_x}\|\Phi_3\|_{H^3_x}+C_0\|\tilde{n}e^{g_2}e^{-\tilde{c}e^{-g_1}}\nabla\phi\|_{H^3_x}
 \end{split}
 \enn
 for each $t\in (0,T^*]$. Thus
 \ben
 \begin{split}
 L_2\leq& (\|\na P\|_{H^3_x}+\|\tilde{n}e^{g_2}e^{-\tilde{c}e^{-g_1}}\na \phi)\|_{H^3_x})\|\Phi_3\|_{H^3_x}\\
 \leq &C_0\|u\|_{H^3_x}\|\Phi_3\|_{H^3_x}^2+C_0\|\tilde{n}e^{g_2}e^{-\tilde{c}e^{-g_1}}\nabla\phi\|_{H^3_x}\|\Phi_3\|_{H^3_x}.
 \end{split}
 \enn
 The Sobolev embedding inequality leads to 
 \ben
 L_1\leq C_0\|u\|_{H^3_x}\|\Phi_3\|_{H^3_x}^2.
 \enn
 Substituting the above estimates on $L_1$ and $L_2$ into \eqref{b8} and summing the results with respect to $0\leq |\alpha|\leq 3$, one gets
 \ben
 \frac{d}{dt}\|\Phi_3\|_{H^3_x}
 \leq C_0\|u\|_{H^3_x}\|\Phi_3\|_{H^3_x}+C_0\|\tilde{n}e^{g_2}e^{-\tilde{c}e^{-g_1}}\nabla\phi\|_{H^3_x},
 \enn
 which, in conjunction with Gronwall's inequality, \eqref{b4} and the following fact: 
 \ben
 \begin{split}
 &\|\tilde{n}e^{g_2}e^{-\tilde{c}e^{-g_1}}\nabla\phi\|_{L^1_T{H^3_x}}\\
 \leq &
 C_0T^{\frac{1}{2}}\|\tilde{n}\|_{L^\infty_TH^2_x}\,e^{C_0\|\tilde{c}\|_{L^\infty_TL^\infty_x}}\|\tilde{c}\|_{L^2_TH^3_x}
 +C_0T^{\frac{1}{2}}\|\tilde{n}\|_{L^2_TH^3_x}\,e^{C_0\|\tilde{c}\|_{L^\infty_TL^\infty_x}}(1+\|\tilde{c}\|_{L^\infty_TH^2_x}^2)\\
 \leq &C_0T^{\frac{1}{2}}\|\tilde{n}\|_{L^\infty_TH^2_x}\,e^{C_0\|\tilde{c}\|_{L^\infty_TH^2_x}}\|\tilde{c}\|_{L^2_TH^3_x}
 +C_0T^{\frac{1}{2}}\|\tilde{n}\|_{L^2_TH^3_x}\,e^{C_0\|\tilde{c}\|_{L^\infty_TH^2_x}}(1+\|\tilde{c}\|_{L^\infty_TH^2_x}^2)
 \end{split}
 \enn
 gives rise to
 \be\label{b9}
 \|\Phi_3\|_{L^\infty_TH^3_x}\leq C_0e^{T\|(\tilde{n},\tilde{c},u)\|_{X_T}}
 +C_0T^{\frac{1}{2}}(\|(\tilde{n},\tilde{c},u)\|_{X_T}+\|(\tilde{n},\tilde{c},u)\|_{X_T}^3)e^{C_0(1+T)\|(\tilde{n},\tilde{c},u)\|_{X_T}}.
 \ee
Let $B_R$ be the closed ball in $X_T$ centered at $0$ with radius $R>0$. Then combining \eqref{b6}, \eqref{b7} and \eqref{b9} we have seen that $\Phi=(\Phi_1,\Phi_2,\Phi_3)$ maps $B_R$ into itself if we first pick $R$ large enough (depending initial and boundary data) and then take $T>0$ sufficiently small (depending on $R$). By a straightforward adaptation of the above arguments, one can also deduce a Lipschitz estimate on $\Phi$, with a Lipschitz constant $C(R)T^\beta$ for some $\beta>0$ and constant $C(R)>0$ depending on $R$. Thus by taking $T<C(R)^{-1/\beta}$, we conclude that $\Phi$ is a contraction on $B_R$ and there exists $(\tilde{n},\tilde{c},u)\in B_R$ and an associated $\na p\in L^\infty(0,T;H^2(\Omega))\cap L^2(0,T;H^3(\Omega))$ (see e.g. \cite{ladyzhenskaya1969}), such that $\Phi(\tilde{n},\tilde{c},u)=(\tilde{n},\tilde{c},u)$, which is a solution of the initial boundary value problem \eqref{e01}.
~\\
\textbf{Local well-posedness of \eqref{cns}-\eqref{bc}.} With the local solution $(\tilde{n},\tilde{c},u,\na p)\in B_R\times (L^\infty(0,T;H^2(\Omega))\\ \cap L^2(0,T;H^3(\Omega)))$ of \eqref{e01}, it is straightforward to verify that $(n,c,u, \na p)\in X_T\times (L^\infty(0,T;H^2(\Omega))\\ \cap L^2(0,T;H^3(\Omega)))$ given by \eqref{b10} is a solution to \eqref{cns}-\eqref{bc}. The regularity \eqref{b11} follows from \eqref{b4} and the non-negativity of $n$ and $c$ follows from the maximum principle. Uniqueness of the solutions can be justified by a similar argument used in \cite{winkler2012}.  

\endProof
~\\
~\\ 
\noindent \textbf{Data Availability}. Data sharing is not applicable to this article as no datasets were generated or analysed during the current study.
~\\
\noindent \textbf{Declaration}. There are no relevant financial or non-financial competing interests to report.
~\\
\noindent \textbf{Acknowledgement}. This work is supported by National Natural Science Foundation of China (No. 12471195) and Heilongjiang Provincial Natural Science Foundation of China (No. YQ2024A001).

\bibliography{rf}
\bibliographystyle{plain}

\end{document}